\documentclass[11pt,reqno]{amsart}
\usepackage{amssymb, amscd, amsmath, color, enumitem, xfrac}
\usepackage{mathrsfs, bm, cite}
\usepackage{graphics}
\usepackage{tikz}
\usepackage{pgfplots}
\usepackage{pgfplotstable}
\usepackage{xcolor}
\definecolor{darkgreen}{rgb}{0,0.5,0}
\usepackage{tkz-fct}
\usetikzlibrary{pgfplots.polar}
\pgfplotsset{compat=newest}
\usepgfplotslibrary{fillbetween}
\usetikzlibrary{patterns}
\input epsf

\numberwithin{equation}{section}
\newtheorem{theorem}{Theorem}[section]
\newtheorem{lemma}[theorem]{Lemma}
\newtheorem{corollary}[theorem]{Corollary}

\theoremstyle{definition}
\newtheorem{definition}[theorem]{Definition}
\newtheorem{example}[theorem]{Example}
\theoremstyle{remark}
\newtheorem{remark}[theorem]{Remark}

\def\XXint#1#2#3{{\setbox0=\hbox{$#1{#2#3}{\int}$ }
\vcenter{\hbox{$#2#3$ }}\kern-.6\wd0}}

\usepackage{hyperref}

\begin{document}
\title[On the Luxemburg Functional in the Concave Regime]{On the Luxemburg Functional in the Concave Regime with an Application to Nonlocal Differential Equations}
\author[C. S. Goodrich]{Christopher S. Goodrich}
%\address{Department of Mathematics\\
%Creighton Preparatory School\\
%Omaha, NE 68114 USA}
\address{School of Mathematics and Statistics\\
UNSW Sydney\\
Sydney, NSW 2052 Australia}
\email[Christopher S. Goodrich]{c.goodrich@unsw.edu.au}
%\urladdr{http://www.math.unl.edu/~s-cgoodri4/}
\author[G. Nakhl]{Gabriel Nakhl}
\address{School of Mathematics and Statistics\\
UNSW Sydney\\
Sydney, NSW 2052 Australia}
\email[Gabriel Nakhl]{g.nakhl@student.unsw.edu.au}
%\date{June 21, 2026}
%\dedicatory{The first author wishes to dedicate this paper to the memory of his beloved miniature dachshund Maddie Goodrich (16 March 2002 - 16 March 2020).}

\begin{abstract}
For $p\ : \ [0,1]\rightarrow(0,1)$ a continuous function, we consider the Luxemburg functional
\begin{equation}
\Lambda_{p(\cdot)}(u):=\inf\left\{\delta>0\ : \ \int_0^1\left|\frac{u(t)}{\delta}\right|^{p(t)}\ dt\le1\right\}.\notag
\end{equation}
We develop a suitable functional analytic framework for $\Lambda_{p(\cdot)}$ so that it can be incorporated into a positive order cone, which can then be used in the analysis of nonlocal differential equations having the general form
\begin{equation}
-A\!\left(\int_0^1 b(1-s)\big(u(s)\big)^{p(s)}\,ds\right)u''(t)=\lambda f\bigl(t,u(t)\bigr),\qquad t\in(0,1).\notag
\end{equation}
By means of a specific numerical example, we demonstrate that, as opposed to a methodology relying only on the supremum norm, our methodology via the Luxemburg functional yields an appreciably sharper localisation of any solutions of the nonlocal equation.
\end{abstract}

\maketitle

\section{Introduction}
Let $p\ : \ [0,1]\rightarrow(0,1)$ be a continuous function, and for $u\in\mathscr{C}\big([0,1]\big)$ consider the Luxemburg functional
\begin{equation}
\Lambda_{p(\cdot)}(u):=\inf\left\{\delta>0\ : \ \int_0^1\left|\frac{u(t)}{\delta}\right|^{p(t)}\ dt\le1\right\}.\notag
\end{equation}
Note that this construction is a one-dimensional analogue of the usual Luxemburg norm
\begin{equation*}
\|u\|_{L^{p(\cdot)}}=
\inf\left\{\delta>0:\int_{\Omega}\left|\frac{u(\bm{x})}{\delta}\right|^{p(\bm{x})}\,d\bm{x}\le1\right\},
\end{equation*}
where, in the latter case, $u\in L^{p(\cdot)}(\Omega)$, with $p(t)>1$ and $\Omega\subset\mathbb{R}^{n}$.  The Luxemburg \emph{norm} has been studied extensively -- see, for example, Diening, et al. \cite{diening_lebesgue_2011} and Khamsi and Kozlowski~\cite{khamsi_fixed_2015}.  And variable-type growth, for which the Luxemburg norm is well suited, occurs frequently in both differential equations and regularity theory -- see, for example, Chen, et al. \cite{chen_mathematical_2006}, Fey and Foss \cite{feyfoss2}, Rajagopal and R\r{u}\v{z}i\v{c}ka \cite{rajagopal1}, R\r{u}\v{z}i\v{c}ka \cite{ruzicka_electrorheological_2000}, Vetro \cite{vetro_variable_2022}, and Zhikov \cite{zhikov1}.

Our goal in this work is to build a suitable functional analytic theory for $u\mapsto\Lambda_{p(\cdot)}(u)$ so as to permit its use in the analysis of nonlocal differential equations of the form
\begin{equation}\label{eq1.1}
-A\left(\left(b*u^{p(\cdot)}\right)(1)\right)u''(t)
=
\lambda f\bigl(t,u(t)\bigr),
\qquad t\in(0,1),
\end{equation}
which will be subject to some given boundary data at $t=0$ and $t=1$, and where
\begin{equation*}
\left(b*u^{p(\cdot)}\right)(1)
:=
\int_0^1 b(1-s)\bigl(u(s)\bigr)^{p(s)}\,ds
\end{equation*}
is a finite convolution; note that the nonlocal coefficient in \eqref{eq1.1} incorporates a genuinely variable-growth quantity, where the variable exponent satisfies, as above $0<p(t)<1$.  Throughout we assume
\[
A:[0,+\infty)\to\mathbb{R}, \qquad
f:[0,1]\times[0,+\infty)\to[0,+\infty), \qquad
p:[0,1]\to(0,1)
\]
are continuous, and that $b\in L^1((0,1])$ is almost everywhere nonnegative.  The precise assumptions are stated at the beginning of Section 2.  Also, the number $\lambda>0$ occurring in \eqref{eq1.1} is a parameter, which characterises how the size of the forcing term $f$ affects the existence theory presented in Section 3.  As we will discuss in more detail toward the end of this introduction, there is a very large literature on nonlocal differential equations.  In particular, we emphasise that our nonlocal coefficient $A$ can be very general as, besides continuity, we only require that it be positive on a possibly very small measure that is not \emph{a priori} fixed.

Note that the convolutional form in \eqref{eq1.1} is broad enough to encompass a number of nonlocal effects. In particular, when $b(t)\equiv1$ the nonlocal term reduces to
\[
\int_0^1 \bigl(u(s)\bigr)^{p(s)}\,ds,
\]
and when, furthermore, $p(t)\equiv1$ we recover the familiar constant-growth Kirchhoff-type setting. More generally, setting $\displaystyle b(t):=\frac{1}{\Gamma(\alpha)}t^{\alpha-1}$, $\alpha>0$, permits Riemann-Liouville fractional integral nonlocal elements -- cf., \cite{goodrich_existence_2010,lan2,lan_compactness_2020}.

Now, as suggested in the first paragraph, in the range $0<p(t)<1$ the Luxemburg functional remains well defined, but it is no longer a norm in the Banach-space sense. Consequently, the analysis must be built on modular estimates and order-theoretic structure rather than on the full linear geometry available in the convex case. 

Recent work on H\"older-type inequalities in variable-exponent Lebesgue spaces shows that useful modular estimates remain available in quasinorm settings, including cases in which the exponent is allowed to fall below one \cite{burenkov_tararykova_holder}. This provides useful analytic background for the present concave regime. The main issue here, however, is not merely the availability of ambient H\"older-type estimates. Rather, we seek to build a hybrid cone in which enough of the classical order-theoretic and additive structure survives to permit the usual fixed point index framework to be retained. Thus, the novelty of the present approach lies in transporting sufficient supremum-norm geometry into the concave Luxemburg-type setting, rather than replacing the fixed point framework by a purely modular one.

A central difficulty in the concave regime is that direct supremum-norm estimates do not interact particularly naturally with variable-growth structure. Our point of departure, therefore, is to retain the order-theoretic advantages of a cone framework while replacing a traditionally supremum-based viewpoint by one adapted to the Luxemburg functional. This leads to the use of the collection
\[
\mathscr{A}_{\mathrm{hybrid}}
:=
\left\{
u\in\mathscr{C}([0,1]) :
\begin{aligned}
&u\ge0,\\
&\min_{t\in[\alpha,\beta]}u(t)\ge \eta_0\|u\|_\infty,\\
&(\bm{1}*u)(1)\ge \widetilde C_0\Lambda_{p(\cdot)}(u)
\end{aligned}
\right\},
\]
where $\bm{1}$ denotes the constant map $\bm{1}\colon\mathbb{R}\to\{1\}$, $0<\eta_0\le1$ and $0<\widetilde C_0<1$ are structural constants fixed precisely later in Theorem~\ref{thm:coercivity-calibration}, and $0\le\alpha<\beta\le1$. We use $\mathscr{A}_{\text{hybrid}}$ in our analysis of \eqref{eq1.1}.  Note that the Luxemburg functional coercivity relation $(\bm{1}*u)(1)\ge \widetilde C_0\Lambda_{p(\cdot)}(u)$ plays a central role in the efficacy of our approach.

We recently used in \cite{goodrich2025luxemburgnormlocalisationnonlocal} a set very similar to $\mathscr{A}_{\text{hybrid}}$ in the analysis of \eqref{eq1.1} in the \emph{convex} regime.  Due to the fact that \cite{goodrich2025luxemburgnormlocalisationnonlocal} was centred around the convex regime, i.e., $p(t)>1$, the usual Luxemburg norm structure was readily available to us.  Thus, a principal contribution of this present work is to demonstrate that by further developing the functional analytic framework associated to $\Lambda_{p(\cdot)}$, the results of \cite{goodrich2025luxemburgnormlocalisationnonlocal} can be transplanted into the concave setting.

\begin{figure}
\begin{minipage}[t]{0.495\linewidth}
\vspace{0pt}
\centering
% Push this one down if needed:
\vspace{3mm}
\includegraphics[width=\linewidth]{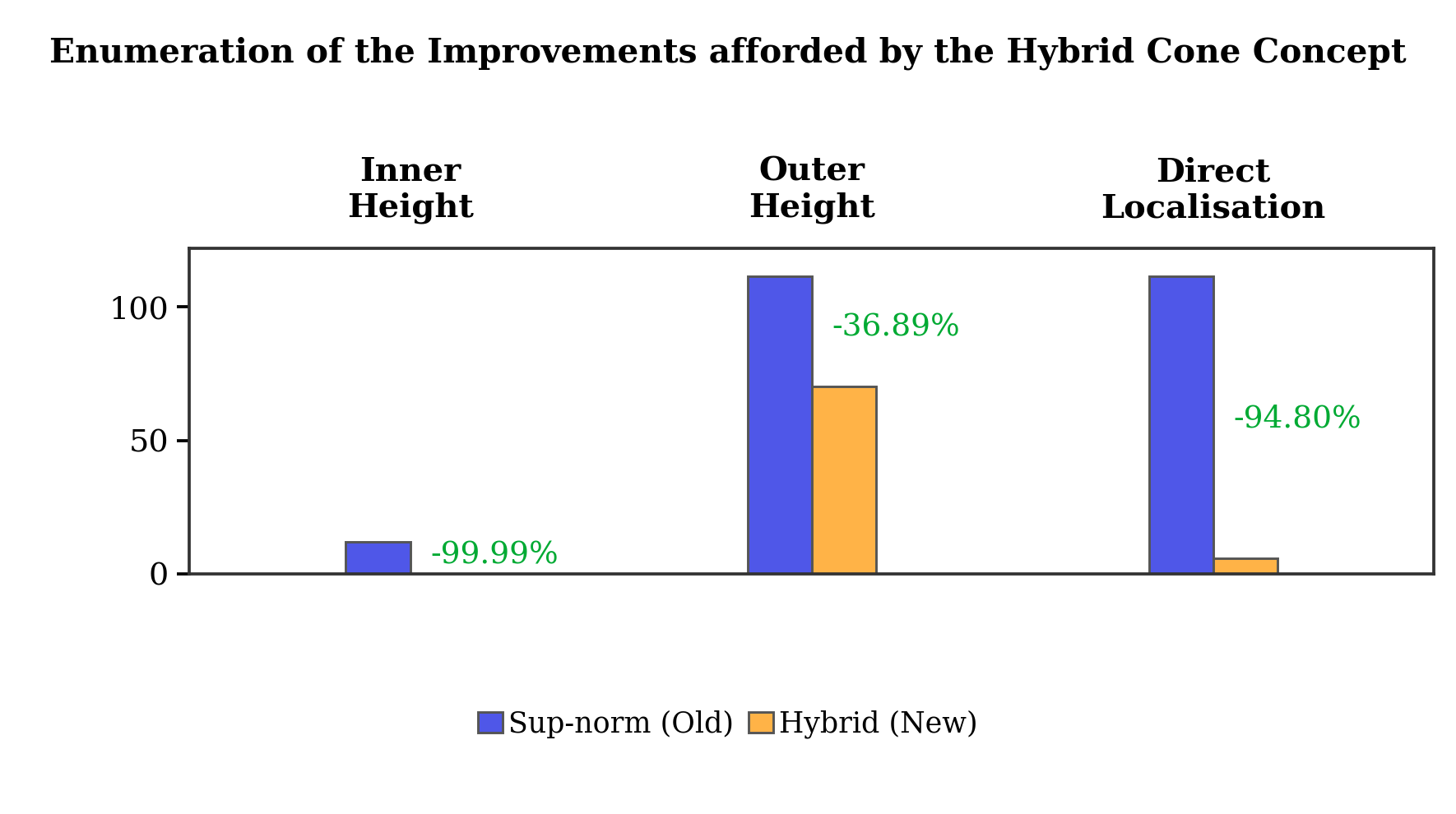}
\caption{A comparison of the outer/inner height and direct localisation in Example \ref{example2.12}.}
\label{fig1}
\end{minipage}
\hfill
\begin{minipage}[t]{0.495\linewidth}
\vspace{0pt}
\centering
% Pull this one up if needed:
% \vspace{-2mm}
\includegraphics[width=\linewidth]{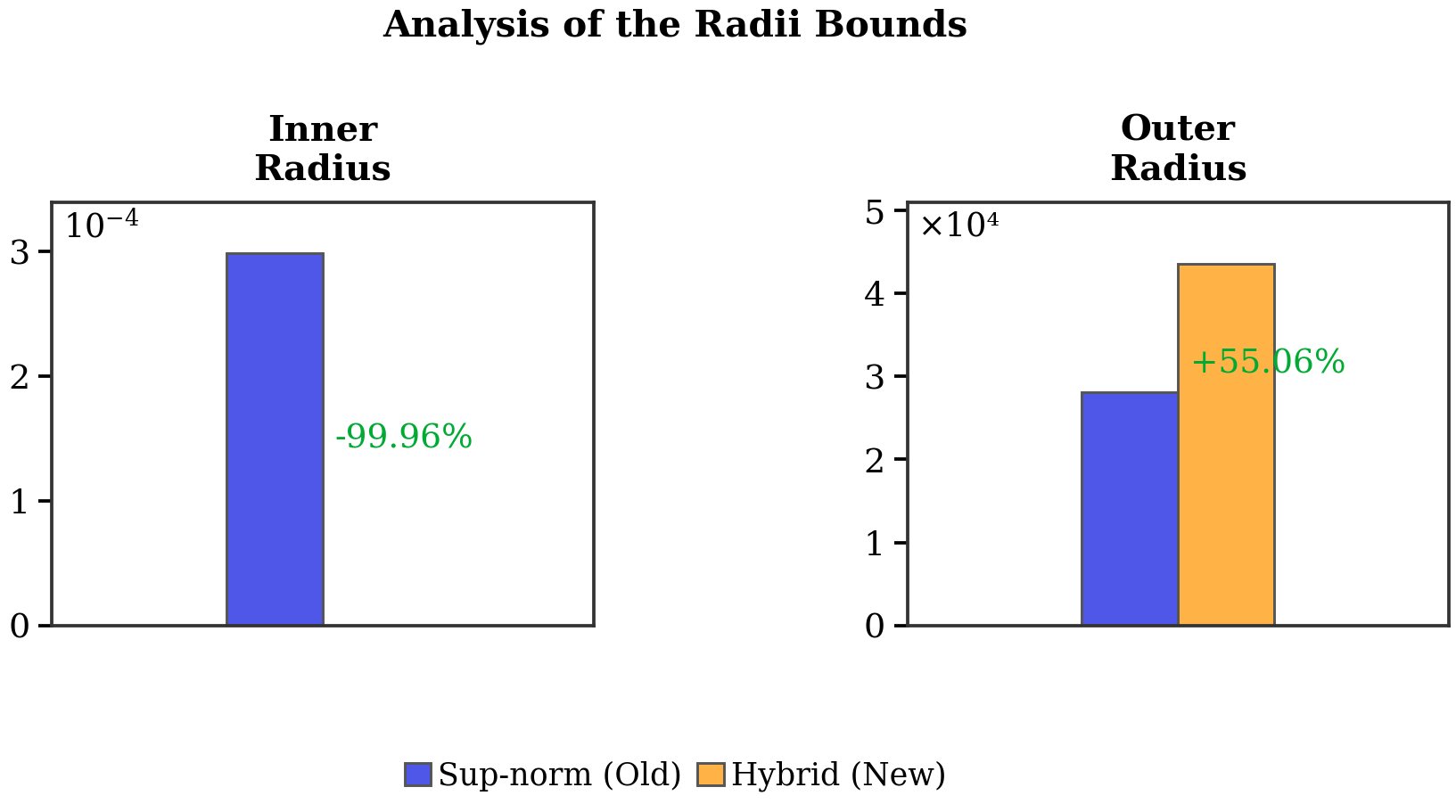}
\vspace{-3.85mm}
\caption{A comparison of the outer/inner radius in Example \ref{example2.12}.  For the outer radius, an \emph{increase} is, in fact, favourable.}
\label{fig2}
\end{minipage}
\end{figure}

We would like to further point out that analysing \eqref{eq1.1} through the lens of the Luxemburg functional provides substantial improvements, not merely in the localisation of solutions but also, perhaps surprisingly, in the conditions required to deduce abstractly existence of solution in the first place.  For example, in Example~\ref{example2.12}, which is a concave modification of \cite[Example~2.13]{goodrich2025unified}, the older concave framework gives existence under the conditions \eqref{eq:old-inner}--\eqref{eq:old-outer}, together with the supremum-norm localisation
\[
  \|u_0\|_{\infty} \le \overline{M}_{\rho_2} \approx 111.4305,
\]
whereas the present Luxemburg framework gives existence under the appreciably sharper conditions \eqref{eq:new-inner}--\eqref{eq:new-outer}, together with the Luxemburg-gauge localisation
\[
\Lambda_{p(\cdot)}\left(u_0\right) \le 3^{8/5} \approx 5.7995.
\]
Figures~\ref{fig1}--\ref{fig2} visually demonstrate the degree of improvement afforded by the present theory. In the example considered, the \textquotedblleft inner rectangle height\textquotedblright\ is reduced by approximately $99.99\%$, the \textquotedblleft outer rectangle height\textquotedblright\ by approximately $36.89\%$, and the direct localisation by approximately $94.80\%$, corresponding to a localisation about $19\times$ tighter. Moreover, the \textquotedblleft inner radius threshold\textquotedblright\ is reduced by approximately $99.96\%$, roughly a $2318\times$ improvement, while the \textquotedblleft outer radius\textquotedblright\ allowance increases by approximately $55.06\%$ -- see Example \ref{example2.12} for a precise description of the terminology \textquotedblleft outer rectangle height\textquotedblright\ and so forth.

In summary, the main theme of the paper is that, even in the strictly concave regime, one can develop a localisation theory that passes through variable-exponent modular control before returning to the supremum norm. A further point is that, although the concave regime does not retain the full Banach-norm geometry of the convex case, enough modular comparison survives to support a useful localisation theory. In particular, the argument developed here combines reverse H\"older-type estimates, lower bounds coming from a thickness condition to be defined below, and quasi-norm control of the intermediate $L^{p(\cdot)/q}$-quantity. This produces a conceptually cleaner upper-bound mechanism and, in at least some settings, a sharper bound than that obtained by the older direct supremum-based approach.

We conclude, as promised at the beginning of this section, by discussing some of the existing literature on nonlocal differential equations.  As was suggested in the opening paragraph of this section, there is a vast literature on nonlocal differential equations.  Often these take either the form
\begin{equation}\label{eq1.2}
-A\left(\Vert u\Vert_{L^p}^{p}\right)u''(t)=\lambda f\big(t,u(t)\big)\text{, }0<t<1
\end{equation}
or the form
\begin{equation}\label{eq1.3}
-A\left(\Vert u'\Vert_{L^p}^{p}\right)u''(t)=\lambda f\big(t,u(t)\big)\text{, }0<t<1
\end{equation}
together with the obvious elliptic PDE equivalents.  One-dimensional nonlocal equations exhibiting variable exponents are much less common, and we are only aware of those by the authors \cite{goodrich_px-growth_2025,goodrich27,goodrich2025unified,goodrich2025luxemburgnormlocalisationnonlocal} and Garc\'{i}a-Huidobro, et al. \cite{garcia1,garcia2}.

In nonlocal problems, it is very common to assume that the nonlocal coefficient, i.e., the function $A$ in \eqref{eq1.2}--\eqref{eq1.3}, satisfies a uniform positivity condition of the form $A(t)>0$ for all $t\ge0$.  Indeed, such an assumption clearly avoids any degeneracy in the equation.  Examples of this type of assumption may be found in Alves and Covei \cite{alves_existence_2015}, Azzouz and Bensedik \cite{azzouz_existence_2012}, Bellamouchi and Zaouche \cite{bellamouchi1}, Biagi, et al. \cite{biagi1,biagi2}, Boulaaras and Guefaifia \cite{boulaaras1}, Calamai and Infante \cite{calamai0,calamai1}, Corr\^{e}a, et al. \cite{correa1,correa2}, do \'{O}, et al. \cite{do4}, Graef and Heidarkhani \cite{graef2}, Infante \cite{infante0,infante2,infante99}, Li, et al. \cite{li1}, Sta\'{n}czy \cite{stanczy1}, and Wang, et al. \cite{wang1}.  Furthermore, this positivity condition is sometimes further adjoined to an assumption of monotonicity -- see Yan and Ma \cite{yanma1} and Yan and Wang \cite{yan1}.  Very occasionally, one finds a relaxation of these conditions such as in the work of Ambrosetti and Arcoya \cite{ambrosetti1}, who assume a vanishing condition at $+\infty$.  Similarly, Delgado, et al. \cite{delgado_non-local_2020} allow for the possibility that the nonlocal coefficient vanishes at a specified point.  And Santos J\'{u}nior and Siciliano \cite{junior_positive_2018-1}, in the case of specifically an $L^2$-norm nonlocal element, assume, in part, only that the coefficient is positive on a neighbourhood of zero.

In 2021 the first author \cite{goodrich_topological_2021-1} introduced a new methodology for treating problems such as \eqref{eq1.2}--\eqref{eq1.3}.  This new methodology employs carefully tailored positive order cones in order to avoid the assumptions mentioned in the previous paragraph.  This methodology subsequently has been extended by Castillo, et al. \cite{castillo1}, Chen and Goodrich \cite{chengoodrich1}, Chu and Hao \cite{chu_positive_2025}, Goodrich, et al. \cite{goodrich_topological_2021-2,goodrich20,goodrich_application_2024-1,goodrich_existence_2022}, Gu and Hao \cite{gu1}, Hao and Gu \cite{haogu1}, Hao and Wang \cite{hao1}, Shibata \cite{shibata3}, and Song and Hao \cite{song1}.  In particular, the methodology introduced in \cite{goodrich_topological_2021-1} does not require that the nonlocal coefficient satisfy any fixed positivity, monotonicity, or asymptotic conditions.  Other than continuity, essentially all that is required is that there be some interval, possibly very small in measure and freely localisable, on which the coefficient is positive.

Thus, in this paper we show that by further developing the notion of the Luxemburg functional, we can extend the good aspects of \cite{goodrich_topological_2021-1,goodrich2025luxemburgnormlocalisationnonlocal} to the concave variable growth setting as embodied by \eqref{eq1.1}.  Moreover, since we are not aware of any general theory for nonlocal differential equations with variable growth in the \emph{concave} regime, it is hoped that these results will spur on additional research in this relatively new area.

\section{Preliminaries and the Hybrid Cone Framework}

\label{sec:preliminaries}

\subsection{General Assumptions}
We work in the Banach space $\mathscr C([0,1])$ under the supremum norm
$\|\cdot\|_\infty$, and adopt the notation and conventions of
\cite[\S2]{goodrich2025luxemburgnormlocalisationnonlocal}. In particular,
$\bm 1$ and $\bm 0$ denote the constant functions equal to $1$ and $0$;
$*$ denotes the finite convolution
\[
(u*v)(t):=\int_0^t u(t-s)v(s)\,ds,\qquad 0\le t\le1,\quad u,v\in L^1((0,1));
\]
and for continuous $h:[0,1]\times[0,+\infty)\to[0,+\infty)$ and reals
$0\le a<b\le1$ and $0\le c<d<+\infty$, we write $h^m_{[a,b]\times[c,d]}$ and
$h^M_{[a,b]\times[c,d]}$ for the minimum and maximum of $h$ over
$[a,b]\times[c,d]$. As in \cite{goodrich2025luxemburgnormlocalisationnonlocal}, solutions of \eqref{eq1.1} are subjected to boundary data at $t=0$ and $t=1$ via a
Green's function $G:[0,1]\times[0,1]\to[0,+\infty)$; our running example is the
Dirichlet kernel
\begin{equation}\label{eq:G_dirichlet}
G(t,s):=
\begin{cases}
t(1-s), & 0\le t\le s\le1,\\
s(1-t), & 0\le s\le t\le1,
\end{cases}
\end{equation}
though the right-focal kernel and other standard choices are allowed in our theory (see \cite{infante_existence_2014,goodrich_new_2018}).

We next state the hypotheses we assume throughout the remainder of this paper.  Note that hypotheses (H1) and (H3) also occur in
\cite[\S2]{goodrich2025luxemburgnormlocalisationnonlocal}.
\begin{list}{}{\setlength{\leftmargin}{.5in}\setlength{\rightmargin}{0in}}
\item[\textbf{H1:}]
The functions $A:[0,+\infty)\to\mathbb{R}$,
$f:[0,1]\times[0,+\infty)\to[0,+\infty)$, and $b:(0,1]\to[0,+\infty)$ satisfy:
\begin{enumerate}
   \item $A$ and $f$ are continuous.
   \item $b\in L^1((0,1])$ and $b(t)>0$ for almost every $t\in(0,1]$.
   \item There exist $0<\rho_1<\rho_2$ such that $A(t)>0$ for all
   $t\in[\rho_1,\rho_2]$.
\end{enumerate}
\item[\textbf{H2:}]
The exponent $p:[0,1]\to(0,1)$, which occurs in \eqref{eq1.1}, is continuous and satisfies
\[
0<p^-:=\min_{t\in[0,1]}p(t)\le p(t)\le p^+:=\max_{t\in[0,1]}p(t)<1.
\]
\item[\textbf{H3:}]
The Green's function $G:[0,1]\times[0,1]\to[0,+\infty)$ satisfies the following.
\begin{enumerate}
   \item Let $\displaystyle\mathscr{G}(s):=\max_{t\in[0,1]}G(t,s)$. Then
   $\mathscr{G}(s)>0$ for all $s\in(0,1)$, and there exist $0\le\alpha<\beta\le1$
   and $\eta_0=\eta_0(\alpha,\beta)\in(0,1]$ such that
   \[
   \min_{t\in[\alpha,\beta]}G(t,s)\ge\eta_0\mathscr{G}(s),\qquad s\in[0,1].
   \]
   \item Define $\displaystyle C_0:=\inf_{s\in(0,1)}\frac{1}{\mathscr{G}(s)}\int_0^1G(t,s)\,dt$.
   Then $0<C_0<1$.
\end{enumerate}
\end{list}

We next state a definition, to which we will return several times in the sequel.

\begin{definition}\label{def:thickness}
A function \(u\in\mathscr C([0,1])\) satisfies the \emph{thickness condition}
if \(u\ge0\) and
\begin{equation}\label{eq:thickness}
\min_{t\in[\alpha,\beta]}u(t)\ge \eta_0\|u\|_\infty.
\end{equation}
\end{definition}

%In the concave hybrid framework below, the coercivity constant appearing in the Luxemburg-type scaling condition will be denoted by $\widetilde C_0$. We do not fix its admissible size at this stage. Instead, its allowable range will be specified later, after the relevant concave comparison estimates have been established. Thus, $\widetilde C_0$ is not introduced as an additional arbitrary assumption, but as a structural constant determined by the estimates developed below.

We define the modular associated with $p$ by
\[
I_{p(\cdot)}(u)
:=
\int_0^1 |u(t)|^{p(t)}\,dt,
\]
and we now introduce a Luxemburg-type modular-induced scaling functional
associated with this modular. As suggested in Section 1, for a measurable function
$u:[0,1]\to\mathbb{R}$, define
\[
\Lambda_{p(\cdot)}(u)
:=
\inf\left\{
\delta>0:
\int_0^1
\left|\frac{u(t)}{\delta}\right|^{p(t)}\,dt
\le1
\right\}.
\]
This is the usual Luxemburg expression associated with the modular
$I_{p(\cdot)}$. However, in the present concave regime $0<p(t)<1$, we do not
regard $\Lambda_{p(\cdot)}$ as a norm in the Banach-space sense. On the ambient
space $\mathscr C([0,1])$ it is used only as a modular-induced scaling
functional. Later, after the relevant comparison and additivity estimates have
been established and the admissible class $\mathscr A_{\mathrm{hybrid}}$ in \eqref{Ahybrid} has
been promoted to the cone $\mathscr K_{\mathrm{hybrid}}$
(Theorem~\ref{thm:cone-promotion}), we show that the restriction of
$\Lambda_{p(\cdot)}$ to $\mathscr K_{\mathrm{hybrid}}$ is a cone-restricted
quasi-norm.

The following elementary facts link the modular and the modular-induced
scaling functional. They are the concave counterpart of
\cite[Lemma~2.1]{goodrich2025luxemburgnormlocalisationnonlocal} and, in the
norm case, are classical (see, e.g., \cite[Ch.~2]{diening_lebesgue_2011}). We
include proofs because those sources assume $p(\cdot)\ge 1$, whereas here
$0<p^-\le p^+<1$; the arguments use only $0<p^-\le p^+<+\infty$, together with
continuity and finite measure, and are otherwise unchanged. They will be used
repeatedly.

\begin{lemma}\label{lem:Luxemburg_basic_concave}
Assume \textnormal{(H2)}. Then $\Lambda_{p(\cdot)}(\bm{0})=0$. Moreover, if
$u\in\mathscr C([0,1])$ and $u\not\equiv\bm{0}$, then
\[
\Lambda_{p(\cdot)}(u)<+\infty
\qquad\text{and}\qquad
I_{p(\cdot)}\!\left(\frac{u}{\Lambda_{p(\cdot)}(u)}\right)=1.
\]
In addition, the map
\[
S(\delta):=
\int_0^1
\left|\frac{u(t)}{\delta}\right|^{p(t)}\,dt,
\qquad \delta>0,
\]
is finite, continuous, and strictly decreasing on $(0,+\infty)$, with
$S(\delta)\to+\infty$ as $\delta\downarrow0$ and $S(\delta)\to0$ as
$\delta\to+\infty$.
\end{lemma}

\begin{proof}
For $u\equiv\bm{0}$, the admissible set in the definition of
$\Lambda_{p(\cdot)}$ is all of $(0,+\infty)$.  Therefore, it follows that $\Lambda_{p(\cdot)}(\bm{0})=0$.

Now assume $u\not\equiv\bm{0}$ and set $M:=\|u\|_\infty>0$. Since $\displaystyle\left|\frac{u(t)}{M}\right|\le1$ and $p(t)>0$, we have $\displaystyle\left|\frac{u(t)}{M}\right|^{p(t)}\le1$, and so, $S(M)\le1$. Therefore
$M$ is admissible, and so,
\[
\Lambda_{p(\cdot)}(u)\le M<+\infty.
\]
Fix $\delta_0:=M$. For any $\delta>0$,
\[
\left|\frac{u(t)}{\delta}\right|^{p(t)}
=
\left(\frac{\delta_0}{\delta}\right)^{p(t)}
\left|\frac{u(t)}{\delta_0}\right|^{p(t)}
\le
\max\left\{
\left(\frac{\delta_0}{\delta}\right)^{p^-},
\left(\frac{\delta_0}{\delta}\right)^{p^+}
\right\}
\left|\frac{u(t)}{\delta_0}\right|^{p(t)},
\]
where the right-hand side is integrable. Therefore, $S(\delta)<+\infty$ for every
$\delta>0$.

To show the continuity of $S$, let $\big\{\delta_n\big\}_{n=1}^{\infty}\subset(0,+\infty)$ and suppose that $\delta_n\to\delta$ with
$\delta>0$. For all sufficiently large $n$, there exists $a>0$ such that
$\delta_n\ge a$. Hence,
\[
\left|\frac{u(t)}{\delta_n}\right|^{p(t)}
\le
\left|\frac{u(t)}{a}\right|^{p(t)}
=
\left(\frac{\delta_0}{a}\right)^{p(t)}
\left|\frac{u(t)}{\delta_0}\right|^{p(t)}
\le
\max\left\{
\left(\frac{\delta_0}{a}\right)^{p^-},
\left(\frac{\delta_0}{a}\right)^{p^+}
\right\}
\left|\frac{u(t)}{\delta_0}\right|^{p(t)},
\]
which is integrable. The pointwise convergence
$\displaystyle\left|\frac{u(t)}{\delta_n}\right|^{p(t)}\to\left|\frac{u(t)}{\delta}\right|^{p(t)}$ together with dominated convergence yields $S(\delta_n)\to S(\delta)$, as desired.

We next show that $S$ is strictly decreasing.  So, for $0<\delta_1<\delta_2$ and every $t\in[0,1]$,
\[
\left|\frac{u(t)}{\delta_2}\right|^{p(t)}
\le
\left|\frac{u(t)}{\delta_1}\right|^{p(t)},
\]
with strict inequality whenever $u(t)\ne0$. Since $u$ is continuous and
$u\not\equiv\bm{0}$, the set $\{t:u(t)\ne0\}$ has positive measure, and so, integration
gives $S(\delta_2)<S(\delta_1)$, which proves the desired claim.

On the other hand, by continuity and the fact that $u\not\equiv\bm{0}$, there exist $c>0$ and a set
$E\subset[0,1]$ with $|E|>0$ such that $\big|u(t)\big|\ge c$ on $E$. For $0<\delta<c$
and $t\in E$, it follows that $\displaystyle\left|\frac{u(t)}{\delta}\right|^{p(t)}\ge\left(\frac{c}{\delta}\right)^{p^-}$, and so,
\[
S(\delta)\ge |E|\left(\frac{c}{\delta}\right)^{p^-}\to+\infty
\qquad\text{as }\delta\downarrow0.
\]
For $\delta\ge\delta_0$, where $\delta_0$ is defined exactly as earlier, we have that $\displaystyle\left|\frac{u(t)}{\delta}\right|^{p(t)}\le\left|\frac{u(t)}{\delta_0}\right|^{p(t)}$, which is integrable and tends to $0$ pointwise as $\delta\to+\infty$. The dominated convergence theorem then implies that $S(\delta)\to0$.

In summary, since $S$ is continuous, strictly decreasing, and has range $(0,+\infty)$, there exists a unique $\delta_*>0$ with $S(\delta_*)=1$. Strict
monotonicity then gives
\[
S(\delta)\le1
\quad\Longleftrightarrow\quad
\delta\ge\delta_*,
\]
and so,
\[
\Lambda_{p(\cdot)}(u)
=
\inf\{\delta>0:S(\delta)\le1\}
=
\delta_*,
\]
from which it follows that
\[
I_{p(\cdot)}\!\left(\frac{u}{\Lambda_{p(\cdot)}(u)}\right)
=
S\bigl(\Lambda_{p(\cdot)}(u)\bigr)
=1,
\]
which proves the remaining claim in the statement of the lemma.  And this completes the proof.
\end{proof}

\begin{corollary}\label{cor:mod_scaling_concave}
Assume \textnormal{(H2)} and let $u\in\mathscr C([0,1])$.
\begin{enumerate}
\item If $\Lambda_{p(\cdot)}(u)\ge1$, then
\[
\Lambda_{p(\cdot)}(u)^{p^-}
\le I_{p(\cdot)}(u)
\le \Lambda_{p(\cdot)}(u)^{p^+}.
\]
\item If $0<\Lambda_{p(\cdot)}(u)\le1$, then
\[
\Lambda_{p(\cdot)}(u)^{p^+}
\le I_{p(\cdot)}(u)
\le \Lambda_{p(\cdot)}(u)^{p^-}.
\]
\end{enumerate}
\end{corollary}

\begin{proof}
This is the two-sided concave analogue of
\cite[Corollaries~2.2 and~2.3]{goodrich2025luxemburgnormlocalisationnonlocal},
which in the convex regime were quoted from \cite{diening_lebesgue_2011}; the
proof uses only the ordering $p^-\le p(t)\le p^+$ together with the
normalisation identity of Lemma~\ref{lem:Luxemburg_basic_concave}.
Set $\delta:=\Lambda_{p(\cdot)}(u)$. In each of cases (1) and (2), we have that $\delta>0$, and so, $u\not\equiv\bm{0}$ (since $\Lambda_{p(\cdot)}(\bm{0})=0$ by
Lemma~\ref{lem:Luxemburg_basic_concave}).  Lemma~\ref{lem:Luxemburg_basic_concave} also gives the normalisation identity
\[
\int_0^1\left|\frac{u(t)}{\delta}\right|^{p(t)}\,dt=1,
\]
and for each $t\in[0,1]$ we have that
\[
|u(t)|^{p(t)}
=
\delta^{p(t)}
\left|\frac{u(t)}{\delta}\right|^{p(t)}.
\]
\emph{Case (i): $\delta\ge1$.}
Since $p^-\le p(t)\le p^+$, we have
\(\delta^{p^-}\le\delta^{p(t)}\le\delta^{p^+}\). Hence, for every $t\in[0,1]$,
\[
\delta^{p^-}\left|\frac{u(t)}{\delta}\right|^{p(t)}
\le |u(t)|^{p(t)}
\le \delta^{p^+}\left|\frac{u(t)}{\delta}\right|^{p(t)}.
\]
Integrating over $[0,1]$ and using the normalisation identity gives
\[
\delta^{p^-}\le I_{p(\cdot)}(u)\le \delta^{p^+}.
\]
\emph{Case (ii): $0<\delta\le1$.}
The exponent inequality reverses:
\(\delta^{p^+}\le\delta^{p(t)}\le\delta^{p^-}\). Hence,
\[
\delta^{p^+}\left|\frac{u(t)}{\delta}\right|^{p(t)}
\le |u(t)|^{p(t)}
\le \delta^{p^-}\left|\frac{u(t)}{\delta}\right|^{p(t)},
\]
and integrating gives
\[
\delta^{p^+}\le I_{p(\cdot)}(u)\le \delta^{p^-}.
\]
Substituting $\delta=\Lambda_{p(\cdot)}(u)$ in each case yields the stated
bounds.
\end{proof}

\begin{remark}
At the borderline $\Lambda_{p(\cdot)}(u)=1$, both cases of
Corollary~\ref{cor:mod_scaling_concave} reduce to $I_{p(\cdot)}(u)=1$, in
agreement with the normalisation identity in
Lemma~\ref{lem:Luxemburg_basic_concave}.
\end{remark}

\begin{lemma}\label{lem:gauge-le-sup}
Assume \textnormal{(H2)}. For every \(w\in\mathscr C([0,1])\) with \(w\ge0\), we have that
\[
\Lambda_{p(\cdot)}(w)\le \|w\|_\infty.
\]
\end{lemma}

\begin{proof}
This is the modular-gauge counterpart of the embedding
\cite[Lemma~2.6]{goodrich2025luxemburgnormlocalisationnonlocal}; here only the
one-sided bound for nonnegative $w$ is required, for which the short argument
below suffices.
If \(w\equiv\bm{0}\), then this is immediate. Otherwise, $\displaystyle\left|\frac{w(t)}{\Vert w\Vert_{\infty}}\right|\le1$ for every \(t\in[0,1]\), and since \(p(t)>0\), it follows that
\[
\int_0^1
\left|\frac{w(t)}{\|w\|_\infty}\right|^{p(t)}\,dt
\le1.
\]
Hence, \(\|w\|_\infty\) is admissible in the definition of
\(\Lambda_{p(\cdot)}(w)\), and so, \(\Lambda_{p(\cdot)}(w)\le\|w\|_\infty\).
\end{proof}

\begin{remark}\label{rem:exponent-generic}
Lemma~\ref{lem:Luxemburg_basic_concave},
Corollary~\ref{cor:mod_scaling_concave}, and Lemma~\ref{lem:gauge-le-sup} are
stated under \textnormal{(H2)}, but their proofs use only the bounds
$0<p^-\le p(t)\le p^+<+\infty$. Consequently, each result holds verbatim with
$p(\cdot)$ replaced by any continuous exponent $s:[0,1]\to(0,+\infty)$, with
$p^\pm$ replaced by $s^-:=\min_{t\in[0,1]}s(t)$ and
$s^+:=\max_{t\in[0,1]}s(t)$. We refer to this as the \emph{exponent-generic
form} of these results. In particular, for every $q>1$ they apply to the
divided exponent $s(\cdot)=p(\cdot)/q$ and to the multiplied exponent
$s(\cdot)=q\,p(\cdot)$, both of which occur in the sequel.
\end{remark}

\subsection{The Hybrid Functional Analytic Framework}
We now introduce the admissible hybrid framework used throughout the paper.
Since the Luxemburg-type modular-induced scaling functional
$\Lambda_{p(\cdot)}$ has not yet been shown to possess the quasi-norm
properties required later, we first define an admissible class rather than
immediately referring to it as a cone.
Let $\widetilde C_0:=C_0\eta_0(\beta-\alpha)^{\sfrac{1}{p^-}}>0$ be the concave hybrid coercivity constant of
Theorem~\ref{thm:coercivity-calibration}. Define
\begin{equation}\label{Ahybrid}
\mathscr A_{\mathrm{hybrid}}
:=
\left\{
u\in\mathscr C([0,1]) :
\begin{aligned}
&u\ge0,\\
&\min_{t\in[\alpha,\beta]}u(t)\ge \eta_0\|u\|_\infty,\\
&(\bm 1*u)(1)\ge \widetilde C_0\,\Lambda_{p(\cdot)}(u)
\end{aligned}
\right\}.
\end{equation}
Thus, $\mathscr A_{\mathrm{hybrid}}$ consists of nonnegative continuous
functions satisfying both \eqref{eq:thickness}, together with the Luxemburg-type
scaling coercivity condition $(\bm 1*u)(1)\ge \widetilde C_0\,\Lambda_{p(\cdot)}(u)$. After the quasi-Luxemburg norm structure has
been established, this admissible class will be identified with the modified
hybrid cone and denoted by $\mathscr K_{\mathrm{hybrid}}$.

For $\rho>0$, define
\[
\widehat V_\rho
:=
\left\{
u\in\mathscr A_{\mathrm{hybrid}} :
\left(b*u^{p(\cdot)}\right)(1)<\rho
\right\}.
\]
Equivalently, if
\[
\Phi(u):=
\left(b*u^{p(\cdot)}\right)(1)
=
\int_0^1 b(1-s)\bigl(u(s)\bigr)^{p(s)}\,ds,
\]
then
\[
\widehat V_\rho
=
\left\{
u\in\mathscr A_{\mathrm{hybrid}} :
\Phi(u)<\rho
\right\}.
\]
At this preliminary stage, $\widehat V_\rho$ is understood as a relatively
open subset of the admissible hybrid class $\mathscr A_{\mathrm{hybrid}}$.
After the cone structure has been established and
$\mathscr K_{\mathrm{hybrid}}$ is identified with
$\mathscr A_{\mathrm{hybrid}}$, the same set will be regarded as a relatively
open subset of the modified hybrid cone.

We first record the continuity of $\Phi$ in the topology induced by
$\|\cdot\|_\infty$ on $\mathscr C([0,1])$.  This is the content of Lemma \ref{lem:Phi-continuity-concave}.
\begin{lemma}\label{lem:Phi-continuity-concave}
Assume \textnormal{(H1)} and \textnormal{(H2)}. The map
$\Phi(u):=\left(b*u^{p(\cdot)}\right)(1)$ is continuous on
$\mathscr C([0,1])$ in the $\|\cdot\|_\infty$-topology; in particular, if
$u_n,u\ge0$ and $\|u_n-u\|_\infty\to0$, then $\Phi(u_n)\to\Phi(u)$.
\end{lemma}

\begin{proof}
Uniform convergence bounds the sequence $\big\{u_n\big\}_{n=1}^{\infty}$ by some $M$, whence $0\le u_n^{p(s)}\le\max\{1,M^{p^+}\}$, and the dominating function
$b(1-s)\max\{1,M^{p^+}\}$ lies in $L^1((0,1])$; dominated convergence gives the
claim. The exponent range is used only through the bound $M^{p^+}$, and so, the
argument is unaffected by \textnormal{(H2)}.
\end{proof}

By Lemma~\ref{lem:Phi-continuity-concave}, the set
\[
\widehat V_\rho
=
\mathscr A_{\mathrm{hybrid}}\cap\{u:\Phi(u)<\rho\}
\]
is relatively open in $\mathscr A_{\mathrm{hybrid}}$, where
$\mathscr A_{\mathrm{hybrid}}$ is equipped with the subspace topology inherited
from $\mathscr C([0,1])$.
We next record the elementary scaling property needed to identify the
relative boundary of $\widehat V_\rho$.

\begin{lemma}\label{lem:positive-scaling-admissible}
Let $u\in\mathscr A_{\mathrm{hybrid}}$ and let $\lambda\ge0$. Then
\[
\lambda u\in\mathscr A_{\mathrm{hybrid}}.
\]
Moreover, if $\Phi(u)>0$, then
\[
0<\lambda<1
\quad\Longrightarrow\quad
\Phi(\lambda u)<\Phi(u),
\]
and
\[
\lambda>1
\quad\Longrightarrow\quad
\Phi(\lambda u)>\Phi(u).
\]
\end{lemma}

\begin{proof}
The case $\lambda=0$ is immediate, since $\bm 0\in\mathscr A_{\mathrm{hybrid}}$.
Assume now that $\lambda>0$.
First, $\lambda u\ge0$. Also,
\[
\min_{t\in[\alpha,\beta]}(\lambda u)(t)
=
\lambda\min_{t\in[\alpha,\beta]}u(t)
\ge
\lambda\eta_0\|u\|_\infty
=
\eta_0\|\lambda u\|_\infty.
\]
It remains to check the Luxemburg-type scaling coercivity condition. By
linearity of the convolution,
\[
(\bm 1*(\lambda u))(1)
=
\lambda(\bm 1*u)(1).
\]
We also have
\[
\Lambda_{p(\cdot)}(\lambda u)
=
\lambda\Lambda_{p(\cdot)}(u).
\]
Indeed,
\[
\int_0^1
\left|
\frac{\lambda u(t)}{\delta}
\right|^{p(t)}\,dt
=
\int_0^1
\left|
\frac{u(t)}{\delta/\lambda}
\right|^{p(t)}\,dt,
\]
and taking the infimum over $\delta>0$ is equivalent to taking the infimum
over $\delta/\lambda>0$. Hence, the claimed homogeneity holds.
Therefore,
\[
(\bm 1*(\lambda u))(1)
=
\lambda(\bm 1*u)(1)
\ge
\lambda\widetilde C_0\,\Lambda_{p(\cdot)}(u)
=
\widetilde C_0\,\Lambda_{p(\cdot)}(\lambda u).
\]
Thus, $\lambda u\in\mathscr A_{\mathrm{hybrid}}$.

Finally,
\[
\Phi(\lambda u)
=
\int_0^1 b(1-s)\lambda^{p(s)}u(s)^{p(s)}\,ds.
\]
Since $b\ge0$ a.e. and $\Phi(u)>0$, the set
\[
E:=\left\{s\in[0,1]: b(1-s)u(s)^{p(s)}>0\right\}
\]
has positive measure.
If $0<\lambda<1$, then $\lambda^{p(s)}<1$ for every $s\in[0,1]$. Hence,
\[
b(1-s)\lambda^{p(s)}u(s)^{p(s)}
<
b(1-s)u(s)^{p(s)}
\]
on $E$, while non-strict inequality holds everywhere. Integrating gives
\[
\Phi(\lambda u)<\Phi(u).
\]
Similarly, if $\lambda>1$, then $\lambda^{p(s)}>1$ for every $s\in[0,1]$. Thus,
\[
b(1-s)\lambda^{p(s)}u(s)^{p(s)}
>
b(1-s)u(s)^{p(s)}
\]
on $E$, while non-strict inequality holds everywhere. Integrating gives
\[
\Phi(\lambda u)>\Phi(u),
\]
which completes the proof.
\end{proof}
We now identify the relative boundary of $\widehat V_\rho$ inside
$\mathscr A_{\mathrm{hybrid}}$.

\begin{lemma}\label{lem:boundary-characterisation-concave}
For each $\rho>0$,
\[
\partial\widehat V_\rho
=\{u\in\mathscr A_{\mathrm{hybrid}}:\Phi(u)=\rho\},
\]
which is the boundary taken relative to $\mathscr A_{\mathrm{hybrid}}$.
\end{lemma}

\begin{proof}
($\subseteq$) Let $u\in\partial\widehat V_\rho\subseteq\mathscr A_{\mathrm{hybrid}}$.
Relative-openness of $\widehat V_\rho$ excludes $\Phi(u)<\rho$, and so, $\Phi(u)\ge\rho$. If $\Phi(u)>\rho$, then choose
$\widehat V_\rho\ni u_n\to u$.  So,
$\Phi(u_n)<\rho$ while $\Phi(u_n)\to\Phi(u)$ by
Lemma~\ref{lem:Phi-continuity-concave}, giving $\Phi(u)\le\rho$, which is a
contradiction. Hence, $\Phi(u)=\rho$.

($\supseteq$) Let $u\in\mathscr A_{\mathrm{hybrid}}$ with $\Phi(u)=\rho>0$. By
Lemma~\ref{lem:positive-scaling-admissible}, $\lambda u\in\mathscr A_{\mathrm{hybrid}}$
for all $\lambda\ge0$, with $\Phi(\lambda u)<\rho$ for $0<\lambda<1$, and
$\Phi(\lambda u)>\rho$ for $\lambda>1$. Thus, $\lambda u\in\widehat V_\rho$ for
$\lambda\uparrow1$ and $\lambda u\notin\widehat V_\rho$ for $\lambda\downarrow1$,
while $\lambda u\to u$ in $\|\cdot\|_\infty$ in both cases. Hence, $u$ lies in
the relative boundary of $\widehat V_\rho$.  And this completes the proof.
\end{proof}

Our next three lemmata concern some technical properties of the Luxemburg functional $\Lambda_{p(\cdot)}$.  In particular, Lemma \ref{lem:gauge-lattice-monotone} demonstrates that the Luxemburg functional preserves order and is homogeneous, Lemma \ref{lem:thickness-concave} demonstrates that the functional satisfies a particular relationship with the infinity norm, and Lemma \ref{scaledtriangleinequality} shows that the functional obeys a scaled triangle inequality.  Each of these will prove to be essential later -- e.g., in both the proof of Lemma \ref{lem:Lp_lower_combined} and the proof of Corollary \ref{cor:Vrho-bounded}.  Note that Lemma \ref{lem:thickness-concave} is the analogue of \cite[Lemma 2.10]{goodrich2025luxemburgnormlocalisationnonlocal}.

\begin{lemma}\label{lem:gauge-lattice-monotone}
Let $s:[0,1]\to(0,+\infty)$ be continuous with
$0<s^-\le s(t)\le s^+<+\infty$, and let $v,w\in\mathscr C([0,1])$.
\begin{enumerate}
\item If $|v(t)|\le|w(t)|$ for all $t\in[0,1]$, then
\(\Lambda_{s(\cdot)}(v)\le\Lambda_{s(\cdot)}(w)\).
\item For every $c\ge0$ it holds that \(\Lambda_{s(\cdot)}(cw)=c\,\Lambda_{s(\cdot)}(w)\).
\end{enumerate}
\end{lemma}

\begin{proof}
To prove part (1) we argue as follows.  For every $\delta>0$, since $|v(t)|\le|w(t)|$ and $s(t)>0$, it follows that
\[
 \int_0^1\left|\frac{v(t)}{\delta}\right|^{s(t)}dt
 \le\int_0^1\left|\frac{w(t)}{\delta}\right|^{s(t)}dt,
\]
and so, every $\delta$ admissible for $w$ is admissible for $v$. Taking the infimum of both sides of the inequality yields the desired claim.

On the other hand, to prove part (2), notice first that if $c=0$, then both sides vanish (using that $\Lambda_{s(\cdot)}(\bm 0)=0$).
For $c>0$, on the other hand,
\[
 \int_0^1\left|\frac{c\,w(t)}{\delta}\right|^{s(t)}dt
 =\int_0^1\left|\frac{w(t)}{\delta/c}\right|^{s(t)}dt,
\]
and as $\delta$ ranges over $(0,+\infty)$ so does $\delta/c$; hence, the
admissible set for $cw$ is $c$ times that for $w$.  Therefore, taking the infimum of both sides of the identity gives
$\Lambda_{s(\cdot)}(cw)=c\,\Lambda_{s(\cdot)}(w)$, as desired.
\end{proof}

\begin{lemma}\label{lem:thickness-concave}
Assume \textnormal{(H2)} and \textnormal{(H3)}. Let \(u\in\mathscr C([0,1])\)
satisfy \eqref{eq:thickness}. Then
\[
\Lambda_{p(\cdot)}(u)
\ge
\eta_0(\beta-\alpha)^{1/p^-}\|u\|_\infty.
\]
Consequently,
\[
\|u\|_\infty
\le
\frac{1}{\eta_0(\beta-\alpha)^{1/p^-}}
\Lambda_{p(\cdot)}(u).
\]
\end{lemma}

\begin{proof}
The proof of this lemma is essentially the same as the proof of \cite[Lemma 2.10]{goodrich2025luxemburgnormlocalisationnonlocal}.  First note that for \(u\equiv\bm{0}\), the conclusion is immediate.

Therefore, henceforth assume that
\(u\not\equiv\bm{0}\).
Set
\begin{equation}
m:=\beta-\alpha.
\end{equation}
Since \(0\le\alpha<\beta\le1\), we have \(0<m\le1\). Define
\[
\delta_*:=\eta_0 m^{1/p^-}\|u\|_\infty.
\]
Since \(u\not\equiv\bm{0}\), we have \(\|u\|_\infty>0\), and hence
\(\delta_*>0\).
By the thickness assumption,
\[
u(t)\ge \eta_0\|u\|_\infty
\qquad \text{for every }t\in[\alpha,\beta].
\]
Therefore,
\[
\int_0^1
\left|\frac{u(t)}{\delta_*}\right|^{p(t)}\,dt
\ge
\int_\alpha^\beta
\left(
\frac{\eta_0\|u\|_\infty}
{\eta_0 m^{1/p^-}\|u\|_\infty}
\right)^{p(t)}\,dt
=
\int_\alpha^\beta
m^{-p(t)/p^-}\,dt.
\]

Now, since \(0<m\le1\) and \(p(t)\ge p^-\), we have
\[
m^{-p(t)/p^-}\ge m^{-1}
\qquad \text{for every }t\in[\alpha,\beta].
\]
Thus,
\[
\int_0^1
\left|\frac{u(t)}{\delta_*}\right|^{p(t)}\,dt
\ge
\int_\alpha^\beta m^{-1}\,dt
=
(\beta-\alpha)m^{-1}
=
1,
\]
so that, setting
\[
S_u(\delta)
:=
\int_0^1
\left|\frac{u(t)}{\delta}\right|^{p(t)}\,dt,
\qquad \delta>0,
\]
we have \(S_u(\delta_*)\ge1\).
By Lemma~\ref{lem:Luxemburg_basic_concave}, the map
\(\delta\mapsto S_u(\delta)\) is strictly decreasing on \((0,+\infty)\).
Hence, for every \(0<\delta<\delta_*\),
\[
S_u(\delta)>S_u(\delta_*)\ge1.
\]
Therefore, no \(\delta<\delta_*\) is admissible in the definition of
\(\Lambda_{p(\cdot)}(u)\), and so,
\[
\Lambda_{p(\cdot)}(u)\ge \delta_*
=
\eta_0(\beta-\alpha)^{1/p^-}\|u\|_\infty.
\]
Since \(\eta_0(\beta-\alpha)^{1/p^-}>0\), rearranging gives
\[
\|u\|_\infty
\le
\frac{1}{\eta_0(\beta-\alpha)^{1/p^-}}
\Lambda_{p(\cdot)}(u),
\]
from which the desired inequality follows at once.
\end{proof}

\begin{lemma}\label{scaledtriangleinequality}
Assume \textnormal{(H2)} and \textnormal{(H3)}. Let
\(u,v\in\mathscr C([0,1])\) satisfy \eqref{eq:thickness}. Then
\begin{equation}\label{eq:pseudo_triangle_ineq}
\Lambda_{p(\cdot)}(u+v)
\le
Q_1\bigl(\Lambda_{p(\cdot)}(u)+\Lambda_{p(\cdot)}(v)\bigr),
\end{equation}
where
\begin{equation}\label{eq:Q1-transport-constant}
Q_1:=
\frac{1}{\eta_0(\beta-\alpha)^{1/p^-}}.
\end{equation}
\end{lemma}

\begin{proof}
By Lemma~\ref{lem:thickness-concave} applied to
\(u\) and \(v\) separately, we see both that
\[
\|u\|_\infty
\le
Q_1\Lambda_{p(\cdot)}(u)
\qquad
\text{and that}
\qquad
\|v\|_\infty
\le
Q_1\Lambda_{p(\cdot)}(v).
\]
At the same time, Lemma~\ref{lem:gauge-le-sup} applied to \(u+v\),
\[
\Lambda_{p(\cdot)}(u+v)\le\|u+v\|_\infty.
\]
Combining these with the triangle inequality for the supremum norm,
\[
\Lambda_{p(\cdot)}(u+v)
\le
\|u+v\|_\infty
\le
\|u\|_\infty+\|v\|_\infty
\le
Q_1\bigl(\Lambda_{p(\cdot)}(u)+\Lambda_{p(\cdot)}(v)\bigr),
\]
as claimed.
\end{proof}

\begin{remark}
Inequality~\eqref{eq:pseudo_triangle_ineq} is not an intrinsic triangle
inequality for $\Lambda_{p(\cdot)}$ on the ambient space $\mathscr C([0,1])$.
It is a cone-restricted estimate, obtained by transporting the supremum-norm
triangle inequality through the thickness comparison of
Lemma~\ref{lem:thickness-concave}. The transport constant $Q_1\ge1$ measures
the loss incurred in passing from the supremum-norm geometry to the
Luxemburg-type functional, and the inequality is available only for functions
satisfying \eqref{eq:thickness} on $[\alpha,\beta]$.
\end{remark}

Our next result demonstrates that the constant $C_0\eta_0(\beta-\alpha)^{1/p^-}$ acts as a coercivity constant for the functional $u\mapsto(\bm{1}*u)(1)$ relative to $\Lambda_{p(\cdot)}$.

\begin{theorem}\label{thm:coercivity-calibration}
Assume \textnormal{(H2)} and \textnormal{(H3)}. Let \(Q_1\) be defined by
\eqref{eq:Q1-transport-constant}. Define the constant $\widetilde{C}_0$ by
\begin{equation}\label{eq:C0tilde-calibrated}
\widetilde C_0
:=
\frac{C_0}{Q_1}
=
C_0\eta_0(\beta-\alpha)^{1/p^-}.
\end{equation}
Then
\[
Q_1\widetilde C_0=C_0.
\]
Furthermore,
\[
\widetilde C_0
=
C_0\eta_0(\beta-\alpha)^{1/p^-}
\le
\eta_0(\beta-\alpha).
\]
Consequently, whenever \(u\in\mathscr C([0,1])\) satisfies \eqref{eq:thickness}, we have
\[
(\bm 1*u)(1)
\ge
\eta_0(\beta-\alpha)\|u\|_\infty
\ge
\eta_0(\beta-\alpha)\Lambda_{p(\cdot)}(u)
\ge
\widetilde C_0\Lambda_{p(\cdot)}(u).
\]
\end{theorem}

\begin{proof}
By the definition of \(\widetilde C_0\), it obviously follows that $Q_1\widetilde C_0=C_0$. By the definition of \(Q_1\) it follows that
\[
\widetilde C_0
=
\frac{C_0}{Q_1}
=
C_0\eta_0(\beta-\alpha)^{1/p^-}.
\]
Since \(0<C_0<1\), \(0<\beta-\alpha\le1\), and \(0<p^-<1\), we have
\[
(\beta-\alpha)^{1/p^-}\le \beta-\alpha.
\]
Therefore,
\[
\widetilde C_0
=
C_0\eta_0(\beta-\alpha)^{1/p^-}
\le
\eta_0(\beta-\alpha).
\]
Now let \(u\in\mathscr C([0,1])\) satisfy \eqref{eq:thickness}. Then
\[
(\bm 1*u)(1)
=
\int_0^1 u(t)\,dt
\ge
\int_\alpha^\beta u(t)\,dt
\ge
\eta_0(\beta-\alpha)\|u\|_\infty.
\]
By Lemma~\ref{lem:gauge-le-sup},
\[
\Lambda_{p(\cdot)}(u)\le \|u\|_\infty.
\]
Hence,
\[
(\bm 1*u)(1)
\ge
\eta_0(\beta-\alpha)\|u\|_\infty
\ge
\eta_0(\beta-\alpha)\Lambda_{p(\cdot)}(u)
\ge
\widetilde C_0\Lambda_{p(\cdot)}(u),
\]
which proves the claim.
\end{proof}

Before presenting Theorem \ref{thm:cone-promotion}, we require the following lemma, seeing as Lemma \ref{lem:thickness-additive} is used to prove Theorem \ref{thm:cone-promotion}.  Note that Lemma \ref{lem:thickness-additive} essentially establishes that the thickness condition \eqref{eq:thickness} satisfies an additive closure condition.

\begin{lemma}\label{lem:thickness-additive}
Assume \textnormal{(H3)}. Let \(u,v\in\mathscr C([0,1])\) satisfy \eqref{eq:thickness}. Then \(u+v\) also
satisfies \eqref{eq:thickness}.
\end{lemma}

\begin{proof}
For every \(t\in[\alpha,\beta]\),
\[
u(t)+v(t)
\ge
\eta_0\|u\|_\infty+\eta_0\|v\|_\infty
=
\eta_0\bigl(\|u\|_\infty+\|v\|_\infty\bigr)
\ge
\eta_0\|u+v\|_\infty,
\]
where the last inequality uses the triangle inequality
\(\|u+v\|_\infty\le\|u\|_\infty+\|v\|_\infty\).
\end{proof}

%In this sense, the present cone is the calibrated concave generalisation of the convex hybrid cone: in the convex setting the Luxemburg functional carries the usual triangle inequality and the coercivity constant is \(C_0\), whereas in the concave setting the cone-restricted quasi-triangle inequality carries the transport loss \(Q_1\), and the corresponding coercivity constant becomes \(\widetilde C_0=C_0/Q_1\).
With Lemma \ref{lem:thickness-additive} in hand, our next result, Theorem \ref{thm:cone-promotion},  establishes that $\mathscr{A}_{\text{hybrid}}$ is a cone and, moreover, that when restricted to this cone, the Luxemburg functional $\Lambda_{p(\cdot)}$ is a quasi-norm.

\begin{theorem}\label{thm:cone-promotion}
Assume \textnormal{(H2)} and \textnormal{(H3)}.  Let \(\widetilde C_0\) be defined as in \eqref{eq:C0tilde-calibrated}. Then:
\begin{enumerate}[label=(\roman*)]
\item for every \(u\in\mathscr C([0,1])\) satisfying
\eqref{eq:thickness} we have that
\begin{equation}
(\bm 1*u)(1)\ge\widetilde C_0\,\Lambda_{p(\cdot)}(u).
\end{equation}
\item \(\mathscr A_{\mathrm{hybrid}}\) is a closed pointed convex cone in
\(\mathscr C([0,1])\). (Consequently, henceforth we identify
\(\mathscr K_{\mathrm{hybrid}}:=\mathscr A_{\mathrm{hybrid}}\).)
\item On \(\mathscr K_{\mathrm{hybrid}}\), the Luxemburg-type scaling
functional \(\Lambda_{p(\cdot)}\) is a positively homogeneous cone-restricted
quasi-norm. We denote this restricted functional by
\[
\|u\|_{L^{p(\cdot)}}:=\Lambda_{p(\cdot)}(u),
\qquad u\in\mathscr K_{\mathrm{hybrid}}.
\]
This notation is reserved for elements of \(\mathscr K_{\mathrm{hybrid}}\);
on the ambient space \(\mathscr C([0,1])\), the functional is still denoted by
\(\Lambda_{p(\cdot)}\). For every \(u,v\in\mathscr K_{\mathrm{hybrid}}\) and
every \(\lambda\ge0\), we have
\begin{align*}
\|u\|_{L^{p(\cdot)}} &\ge 0,\\
\|u\|_{L^{p(\cdot)}}&=0
\Longleftrightarrow
u\equiv\bm{0},\\
\|\lambda u\|_{L^{p(\cdot)}}
&=
\lambda\,\|u\|_{L^{p(\cdot)}},\\
\|u+v\|_{L^{p(\cdot)}}
&\le
Q_1\bigl(
\|u\|_{L^{p(\cdot)}}+\|v\|_{L^{p(\cdot)}}
\bigr),
\end{align*}
where \(Q_1\ge1\) is defined in \eqref{eq:Q1-transport-constant}.
\end{enumerate}
\end{theorem}

\begin{proof}
We begin by noting that part (i) has already been proved in Theorem~\ref{thm:coercivity-calibration}.  Therefore, we begin by proving part (ii). We directly verify the cone properties for \(\mathscr A_{\mathrm{hybrid}}\).  So, first let \((u_n)\subset\mathscr A_{\mathrm{hybrid}}\) with
\(\|u_n-u\|_\infty\to0\). Pointwise convergence gives \(u\ge0\). Since
\[
\left|
\min_{[\alpha,\beta]}u_n
-
\min_{[\alpha,\beta]}u
\right|
\le
\|u_n-u\|_\infty
\qquad\text{and}\qquad
\|u_n\|_\infty\to\|u\|_\infty,
\]
passing to the limit in
\[
\min_{[\alpha,\beta]}u_n\ge\eta_0\|u_n\|_\infty
\]
gives
\[
\min_{[\alpha,\beta]}u\ge\eta_0\|u\|_\infty.
\]
Thus \(u\) is nonnegative and satisfies \eqref{eq:thickness}, so by part
(i) it also satisfies the Luxemburg coercivity condition. Hence,
\(u\in\mathscr A_{\mathrm{hybrid}}\).  Also note that if \(u\in\mathscr A_{\mathrm{hybrid}}\) and
\(-u\in\mathscr A_{\mathrm{hybrid}}\), then \(u\ge0\) and \(-u\ge0\), and so,
\(u\equiv\bm{0}\). Thus,
\[
\mathscr A_{\mathrm{hybrid}}\cap(-\mathscr A_{\mathrm{hybrid}})
=
\{\bm 0\}.
\]

On the other hand, for \(\lambda\ge0\) and \(u\in\mathscr A_{\mathrm{hybrid}}\), we have
\(\lambda u\ge0\) and
\[
\min_{t\in[\alpha,\beta]}(\lambda u)(t)
=
\lambda\min_{t\in[\alpha,\beta]}u(t)
\ge
\lambda\eta_0\|u\|_\infty
=
\eta_0\|\lambda u\|_\infty,
\]
and so, \(\lambda u\) is nonnegative and satisfies \eqref{eq:thickness}. By
part (i), \(\lambda u\) satisfies the Luxemburg coercivity condition.  Consequently,
\(\lambda u\in\mathscr A_{\mathrm{hybrid}}\).

Finally, let \(u,v\in\mathscr A_{\mathrm{hybrid}}\). Each is nonnegative and satisfies
\eqref{eq:thickness}, so by Lemma~\ref{lem:thickness-additive}, \(u+v\) is
nonnegative and satisfies
\[
\min_{[\alpha,\beta]}(u+v)
\ge
\eta_0\|u+v\|_\infty.
\]
By part (i), \(u+v\) then satisfies the Luxemburg coercivity condition
\[
\big(\bm 1*(u+v)\big)(1)
\ge
\widetilde C_0\,\Lambda_{p(\cdot)}(u+v).
\]
Hence, \(u+v\in\mathscr A_{\mathrm{hybrid}}\).

Combining the four properties, we conclude that \(\mathscr A_{\mathrm{hybrid}}\) is a closed
pointed convex cone in \(\mathscr C([0,1])\).  Moreover, we identify
\[
\mathscr K_{\mathrm{hybrid}}
:=
\mathscr A_{\mathrm{hybrid}}.
\]
We next address part (iii) in the statement of the theorem.  The nonnegativity of \(\Lambda_{p(\cdot)}\) is immediate from its definition as an
infimum over positive numbers. By Lemma~\ref{lem:Luxemburg_basic_concave}, we have that
\(\Lambda_{p(\cdot)}(u)=0\) when \(u\equiv\bm{0}\); when \(u\not\equiv\bm{0}\) the
admissible threshold is positive, giving \(\Lambda_{p(\cdot)}(u)>0\). Hence,
\[
\Lambda_{p(\cdot)}(u)=0
\iff
u\equiv\bm{0}.
\]
Positive homogeneity is established in
Lemma~\ref{lem:positive-scaling-admissible}. Finally, since every
\(u,v\in\mathscr K_{\mathrm{hybrid}}\) is nonnegative and satisfies \eqref{eq:thickness}, the pseudo-triangle inequality Lemma \ref{scaledtriangleinequality} applies, giving
\[
\Lambda_{p(\cdot)}(u+v)
\le
Q_1\bigl(
\Lambda_{p(\cdot)}(u)
+
\Lambda_{p(\cdot)}(v)
\bigr).
\]
In the notation of the statement of the theorem, this is the quasi-triangle
inequality for \(\|\cdot\|_{L^{p(\cdot)}}\) on
\(\mathscr K_{\mathrm{hybrid}}\).
\end{proof}

Our next lemma establishes a lower bound on $\Vert\cdot\Vert_{L^{p(\cdot)}}$.  Note that Lemma \ref{lem:Lp_lower_combined} is the analogue of \cite[Lemma 2.11]{goodrich2025luxemburgnormlocalisationnonlocal}.

\begin{lemma}
\label{lem:Lp_lower_combined}
Assume \textnormal{(H2)} and \textnormal{(H3)}, and let $\rho>0$. If
$u \in \partial \widehat V_\rho$, then
\begin{equation*}
\|u\|_{L^{p(\cdot)}}\ge
\eta_0(\beta-\alpha)^{\frac{1}{p^-}}
\left[
\left(\frac{\rho}{(b*\bm{1})(1)}\right)^{\frac{1}{p^+}}
+
\varepsilon_1(\rho,b)
\right],
\end{equation*}
where
\[
\varepsilon_1(\rho,b):=
\begin{cases}
\left(\dfrac{\rho}{(b*\bm{1})(1)}\right)^{\frac{1}{p^-}}
-\left(\dfrac{\rho}{(b*\bm{1})(1)}\right)^{\frac{1}{p^+}},
& 0<\dfrac{\rho}{(b*\bm{1})(1)}<1,\\[2ex]
0,
& \dfrac{\rho}{(b*\bm{1})(1)}\ge1.
\end{cases}
\]
\end{lemma}

\begin{proof}
Since $u\in\partial\widehat V_\rho\subseteq\mathscr K_{\mathrm{hybrid}}$,
the cone-restricted notation
$\|u\|_{L^{p(\cdot)}}=\Lambda_{p(\cdot)}(u)$ is valid. Therefore, following exactly the proof of \cite[Lemma 2.11]{goodrich2025luxemburgnormlocalisationnonlocal}, we simply note that, by
Lemma~\ref{lem:thickness-concave},
\begin{equation}
\|u\|_{L^{p(\cdot)}} \ge
\eta_0(\beta-\alpha)^{\frac{1}{p^-}}\|u\|_\infty,\notag
\end{equation}
whereas, by \cite[Lemma 3.3]{goodrich2025unified},
\begin{equation}
\|u\|_\infty \ge
\left(\frac{\rho}{(b*\bm{1})(1)}\right)^{\frac{1}{p^+}}
+
\varepsilon_1(\rho,b),\notag
\end{equation}
and so, combining the two inequalities yields the stated inequality.
\end{proof}

Our next two lemmata, namely Lemma \ref{lem:thickness-divided-exponent} and Lemma \ref{lem:reverse-holder-divided-exponent}, provide quantitative bounds on the Luxemburg functional $\Lambda_{\frac{p(\cdot)}{q}}$ in case, as indicated, the exponent $\displaystyle\frac{p(\cdot)}{q}$ is realised in a fractional form.  These two results will be used, in part, to prove both Lemma \ref{lem:upper-localisation-divided-exponent} and Corollary \ref{cor:Vrho-bounded}.  Note that Lemma \ref{lem:thickness-divided-exponent} is the analogue of \cite[Lemma 2.17]{goodrich2025luxemburgnormlocalisationnonlocal}, whereas Lemma \ref{lem:reverse-holder-divided-exponent} is the analogue of \cite[Lemma 2.18]{goodrich2025luxemburgnormlocalisationnonlocal}.  However, there is an important dissimilarity between Lemma \ref{lem:reverse-holder-divided-exponent} and \cite[Lemma 2.18]{goodrich2025luxemburgnormlocalisationnonlocal} inasmuch as the former provides for a sharper bound in the case where $\displaystyle\frac{\rho}{C_1(q)}\ge1$.

\begin{lemma}\label{lem:thickness-divided-exponent}
Assume \textnormal{(H2)} and \textnormal{(H3)}, and let $q>1$. If
$u\in\mathscr K_{\mathrm{hybrid}}$, then
\begin{equation}\label{eq:divided-thickness-lower}
\Lambda_{\frac{p(\cdot)}{q}}(u)
\ge
\eta_0(\beta-\alpha)^{\frac{q}{p^-}}\|u\|_\infty,
\end{equation}
and, consequently,
\begin{equation}\label{eq:p-to-divided-gauge-comparison}
\eta_0(\beta-\alpha)^{\frac{q}{p^-}}\Lambda_{p(\cdot)}(u)
\le
\Lambda_{\frac{p(\cdot)}{q}}(u).
\end{equation}
\end{lemma}

\begin{proof}
The proof of \eqref{eq:divided-thickness-lower} is that of
\cite[Lemma~2.17]{goodrich2025luxemburgnormlocalisationnonlocal},
to which we refer for the details: with
$\delta_*:=\eta_0(\beta-\alpha)^{q/p^-}\|u\|_\infty$, \eqref{eq:thickness} together with $p(t)\ge p^-$ gives
\begin{equation}
\int_0^1(u/\delta_*)^{p(t)/q}\,dt\ge1,\notag
\end{equation}
and the conclusion follows from the
strict monotonicity of
\begin{equation}
\delta\mapsto\int_0^1(u/\delta)^{p(t)/q}\,dt.\notag
\end{equation}
Two
points differ from the cited setting in \cite[Lemma~2.17]{goodrich2025luxemburgnormlocalisationnonlocal}. First, here $q>1$ places the divided
exponent $p(\cdot)/q$ below one, so the left-hand side is the Luxemburg gauge
$\Lambda_{p(\cdot)/q}$ rather than a norm; the required strict monotonicity follows from the exponent-generic form of
Lemma~\ref{lem:Luxemburg_basic_concave}, applied with
\(s(\cdot)=p(\cdot)/q\). Second,
\eqref{eq:p-to-divided-gauge-comparison} follows from
\eqref{eq:divided-thickness-lower} together with
Lemma~\ref{lem:gauge-le-sup}, which gives
$\Lambda_{p(\cdot)}(u)\le\|u\|_\infty$, which completes the proof of the lemma.
\end{proof}
\begin{lemma}\label{lem:reverse-holder-divided-exponent}
Assume \textnormal{(H1)} and \textnormal{(H2)}. Let $q>1$ be such that
\[
b^{\frac{1}{1-q}}\in L^1((0,1]),
\]
and define
\[
r(t):=\frac{p(t)}{q},
\qquad
C_1(q):=
\left(
\int_0^1 b(1-s)^{\frac{1}{1-q}}\,ds
\right)^{1-q}.
\]
If $u\in\partial\widehat V_\rho$, then
\[
I_{r(\cdot)}(u)
:=
\int_0^1 u(t)^{\frac{p(t)}{q}}\,dt
\le
\left(
\frac{\rho}{C_1(q)}
\right)^{\frac{1}{q}}.
\]
Consequently,
\[
\Lambda_{\frac{p(\cdot)}{q}}(u)
\le
M_{p/q}(\rho),
\]
where
\[
M_{p/q}(\rho)
:=
\begin{cases}
\displaystyle
\left(
\frac{\rho}{C_1(q)}
\right)^{\frac{1}{p^-}},
& \displaystyle \frac{\rho}{C_1(q)}\ge1,\\[2ex]
\displaystyle
\left(
\frac{\rho}{C_1(q)}
\right)^{\frac{1}{p^+}},
& \displaystyle 0<\frac{\rho}{C_1(q)}<1.
\end{cases}
\]
\end{lemma}

\begin{proof}
Since $u\in\partial\widehat V_\rho$,
Lemma~\ref{lem:boundary-characterisation-concave} gives
\[
\rho
=
\Phi(u)
=
\left(b*u^{p(\cdot)}\right)(1)
=
\int_0^1 b(1-s)u(s)^{p(s)}\,ds.
\]
Writing
\[
u(s)^{p(s)}
=
\left(u(s)^{p(s)/q}\right)^q,
\]
we obtain
\[
\rho
=
\int_0^1 b(1-s)\left(u(s)^{p(s)/q}\right)^q\,ds.
\]
By the reverse H\"older inequality,
\[
\int_0^1 b(1-s)\left(u(s)^{p(s)/q}\right)^q\,ds
\ge
\left(
\int_0^1 b(1-s)^{\frac{1}{1-q}}\,ds
\right)^{1-q}
\left(
\int_0^1 u(s)^{p(s)/q}\,ds
\right)^q.
\]
Thus,
\[
\rho
\ge
C_1(q)
\left(
\int_0^1 u(s)^{p(s)/q}\,ds
\right)^q,
\]
and hence
\[
\int_0^1 u(s)^{p(s)/q}\,ds
\le
\left(
\frac{\rho}{C_1(q)}
\right)^{1/q},
\]
which proves the modular estimate.

Now, set
\[
K:=
\left(
\frac{\rho}{C_1(q)}
\right)^{1/q}.
\]
Then $I_{r(\cdot)}(u)\le K$, with $r^-:=p^-/q$ and $r^+:=p^+/q$.  We consider cases based upon the size of the number $K$.

We first consider the case in which $0<K<1$.  Then $I_{r(\cdot)}(u)<1$, so $\Lambda_{r(\cdot)}(u)<1$. The small-regime
modular--Luxemburg comparison gives
$I_{r(\cdot)}(u)\ge\Lambda_{r(\cdot)}(u)^{r^+}$, whence
\[
\Lambda_{r(\cdot)}(u)
\le
K^{1/r^+}
=
K^{q/p^+}
=
\left(
\frac{\rho}{C_1(q)}
\right)^{1/p^+}.
\]

Next we consider the case in which $K\ge1$.  If $\Lambda_{r(\cdot)}(u)<1$, then the small-regime comparison gives
\[
\Lambda_{r(\cdot)}(u)
\le
K^{1/r^+}
\le
K^{1/r^-}
=
\left(
\frac{\rho}{C_1(q)}
\right)^{1/p^-},
\]
using $r^+\ge r^-$ and $K\ge1$. If instead $\Lambda_{r(\cdot)}(u)\ge1$, then the
large-regime comparison gives
$I_{r(\cdot)}(u)\ge\Lambda_{r(\cdot)}(u)^{r^-}$, whence
\[
\Lambda_{r(\cdot)}(u)
\le
K^{1/r^-}
=
K^{q/p^-}
=
\left(
\frac{\rho}{C_1(q)}
\right)^{1/p^-}.
\]

Combining the cases yields the stated bound for
$\Lambda_{p(\cdot)/q}(u)=\Lambda_{r(\cdot)}(u)$.  And this completes the proof.
\end{proof}

The next result establishes an upper bound on $\Vert u\Vert_{\infty}$ provided that $u\in\partial\widehat{V}_{\rho}$.  Note that Lemma~\ref{lem:upper-localisation-divided-exponent} is similar to \cite[Corollary 2.28]{goodrich2025luxemburgnormlocalisationnonlocal}, though, in fact, it is \emph{stronger} since the upper bound that Lemma \ref{lem:upper-localisation-divided-exponent} yields provides tighter control over $\Vert u\Vert_{\infty}$ than is implied by \cite[Corollary 2.28]{goodrich2025luxemburgnormlocalisationnonlocal} in the convex setting.
\begin{lemma}\label{lem:upper-localisation-divided-exponent}
Assume \textnormal{(H1)}, \textnormal{(H2)}, and \textnormal{(H3)}.
Let $q>1$ be such that
\[
 b^{\frac{1}{1-q}}\in L^1((0,1]),
\]
and let $M_{p/q}(\rho)$ be defined as in
Lemma~\ref{lem:reverse-holder-divided-exponent}. Define
\[
 B_{\infty,\rho}^{(q)}
 := \frac{1}{\eta_0(\beta-\alpha)^{q/p^-}}\,M_{p/q}(\rho).
\]
If $u\in\partial\widehat V_\rho$, then
\[
 \|u\|_\infty \le B_{\infty,\rho}^{(q)}.
\]
Equivalently,
\[
 \|u\|_\infty
 \le \frac{1}{\eta_0(\beta-\alpha)^{q/p^-}}\,M_{p/q}(\rho).
\]
\end{lemma}

\begin{proof}
Since $u\in\partial\widehat V_\rho\subseteq\mathscr K_{\mathrm{hybrid}}$,
Lemma~\ref{lem:thickness-divided-exponent} gives
\[
 \eta_0(\beta-\alpha)^{q/p^-}\|u\|_\infty
 \le \Lambda_{\frac{p(\cdot)}{q}}(u).
\]
On the other hand, by Lemma~\ref{lem:reverse-holder-divided-exponent},
\[
 \Lambda_{\frac{p(\cdot)}{q}}(u) \le M_{p/q}(\rho).
\]
Combining the two inequalities yields
\[
 \eta_0(\beta-\alpha)^{q/p^-}\|u\|_\infty \le M_{p/q}(\rho).
\]
Therefore,
\[
 \|u\|_\infty
 \le \frac{1}{\eta_0(\beta-\alpha)^{q/p^-}}\,M_{p/q}(\rho)
 = B_{\infty,\rho}^{(q)},
\]
as desired.
\end{proof}

As a consequence of the preceding lemmata, we have the following corollary, which establishes that the set $\widehat{V}_{\rho}$ is bounded.  This is important for the use of $\widehat{V}_{\rho}$ in the existence results in the next section.

\begin{corollary}
\label{cor:Vrho-bounded}
Assume \textnormal{(H1)}, \textnormal{(H2)}, and \textnormal{(H3)}.
Suppose that, for some $q>1$,
\[
 b^{\frac{1}{1-q}}\in L^1((0,1]).
\]
Then, for each $\rho>0$, the set $\widehat V_\rho$ is bounded in
$\mathscr C([0,1])\cap L^{p(\cdot)}([0,1])$, in the sense that both
$\|u\|_\infty$ and $\Lambda_{p(\cdot)}(u)$ are uniformly bounded for
$u\in\widehat V_\rho$.
\end{corollary}

\begin{proof}
Recall from Lemma~\ref{lem:reverse-holder-divided-exponent} the constant
\[
 C_1(q):=\left(\int_0^1 b(1-s)^{\frac{1}{1-q}}\,ds\right)^{1-q},
\]
and from Lemma~\ref{lem:upper-localisation-divided-exponent} the bound
\[
 B_{\infty,\rho}^{(q)}
 :=\frac{1}{\eta_0(\beta-\alpha)^{q/p^-}}\,M_{p/q}(\rho).
\]
The hypothesis $b^{1/(1-q)}\in L^1((0,1])$ ensures that
\[
 \int_0^1 b(1-s)^{\frac{1}{1-q}}\,ds
\]
is finite. Since $b>0$ a.e., the preceding integral is also positive, and so, lives in
$(0,+\infty)$. Therefore,
\[
 C_1(q)=\left(\int_0^1 b(1-s)^{\frac{1}{1-q}}\,ds\right)^{1-q}
\]
is also both finite and positive, and consequently $M_{p/q}(\rho)<+\infty$ and
$B_{\infty,\rho}^{(q)}<+\infty$.

Fix $\rho>0$ and let $u\in\widehat V_\rho$, so that
$u\in\mathscr K_{\mathrm{hybrid}}$ and $\Phi(u)<\rho$.
If $u\equiv\bm{0}$, then $\|u\|_\infty=0<B_{\infty,\rho}^{(q)}$.
Otherwise $u\not\equiv\bm{0}$, and so, $\|u\|_\infty>0$. Since $b>0$ almost
everywhere and $u$ is continuous, nonnegative, and not identically zero, it follows that
\[
 \Phi(u)=\int_0^1 b(1-s)u(s)^{p(s)}\,ds>0.
\]

Consider the map $\lambda\mapsto\Phi(\lambda u)$ for $\lambda\ge1$. By
Lemma~\ref{lem:Phi-continuity-concave} this map is continuous, and by
Lemma~\ref{lem:positive-scaling-admissible} it is strictly increasing, so
that
\[
 \lambda\ge1
 \quad\Longrightarrow\quad
 \Phi(\lambda u)\ge\Phi(u),
\]
with strict inequality when $\lambda>1$. At $\lambda=1$ we have
$\Phi(u)<\rho$, while for $\lambda\ge1$ the bound $\lambda^{p(s)}\ge\lambda^{p^-}$
gives
\[
\Phi(\lambda u)
=
\int_0^1 b(1-s)\lambda^{p(s)}u(s)^{p(s)}\,ds
\ge
\lambda^{p^-}
\int_0^1 b(1-s)u(s)^{p(s)}\,ds
=
\lambda^{p^-}\Phi(u).
\]
Since \(p^->0\) and \(\Phi(u)>0\), it follows that
\[
\Phi(\lambda u)\to+\infty
\qquad\text{as }\lambda\to+\infty.
\]

Now, by the intermediate value theorem there exists $\lambda^*>1$ with
$\Phi(\lambda^* u)=\rho$. By
Lemma~\ref{lem:positive-scaling-admissible},
$\lambda^* u\in\mathscr K_{\mathrm{hybrid}}$, and since
$\Phi(\lambda^* u)=\rho$,
Lemma~\ref{lem:boundary-characterisation-concave} gives
$\lambda^* u\in\partial\widehat V_\rho$. Applying
Lemma~\ref{lem:upper-localisation-divided-exponent} to $\lambda^* u$, we see that
\[
 \|\lambda^* u\|_\infty\le B_{\infty,\rho}^{(q)}.
\]
Since $\lambda^*>1$, it follows that
\[
 \|u\|_\infty=\frac{1}{\lambda^*}\,\|\lambda^* u\|_\infty
 <\|\lambda^* u\|_\infty\le B_{\infty,\rho}^{(q)}.
\]
In either case $\|u\|_\infty\le B_{\infty,\rho}^{(q)}<+\infty$, and so,
$\widehat V_\rho$ is bounded in $\mathscr C([0,1])$.
For the gauge bound, Lemma~\ref{lem:gauge-le-sup} gives
$\Lambda_{p(\cdot)}(u)\le\|u\|_\infty\le B_{\infty,\rho}^{(q)}$ for every
$u\in\widehat V_\rho$. Hence, $\widehat V_\rho$ is bounded in the
$L^{p(\cdot)}$-Luxemburg gauge as well, and therefore in
$\mathscr C([0,1])\cap L^{p(\cdot)}([0,1])$, as claimed.
\end{proof}

\begin{remark}
Corollary~\ref{cor:Vrho-bounded} is the concave-regime counterpart of
\cite[Corollary 2.19]{goodrich2025luxemburgnormlocalisationnonlocal}.
In contrast to the older direct supremum-norm route, the bound here is
obtained internally, as assembled in
Lemma~\ref{lem:upper-localisation-divided-exponent}.

The localisation estimate of
Lemma~\ref{lem:upper-localisation-divided-exponent} strictly improves the
older concave-regime supremum estimate of
\cite[Lemma~3.4]{goodrich2025unified}, and the improvement is structural.
Both estimates use the same reverse-H\"older bound and the same thickness
transport, differing only in the middle. 
To make the comparison explicit, set
\[
 K:=\left(\frac{\rho}{C_1(q)}\right)^{1/q},
 \qquad
 C_1(q)=\left(\int_0^1 b(1-s)^{\frac{1}{1-q}}\,ds\right)^{1-q},
\]
where $C_1(q)$ is finite and positive under the hypothesis
$b^{1/(1-q)}\in L^1((0,1])$. The older estimate of
\cite[Lemma~3.4]{goodrich2025unified} reads
\[
 \|u\|_\infty
 \le
 \frac{1}{\eta_0}\,(\beta-\alpha)^{-q/p^-}\,(K+1)^{q/p^-},
\]
whereas Lemma~\ref{lem:upper-localisation-divided-exponent} gives
\[
 \|u\|_\infty
 \le
 \frac{1}{\eta_0}\,(\beta-\alpha)^{-q/p^-}\,M_{p/q}(\rho),
 \qquad
 M_{p/q}(\rho)=
 \begin{cases}
   K^{q/p^-}, & K\ge1,\\[1ex]
   K^{q/p^+}, & 0<K<1.
 \end{cases}
\]
Since $x\mapsto x^{q/p^-}$ is strictly increasing on
$[0,+\infty)$, if $K\ge1$, then $K^{q/p^-}<(K+1)^{q/p^-}$, while if $0<K<1$,
then $K^{q/p^+}<1<(K+1)^{q/p^-}$. Hence, $M_{p/q}(\rho)<(K+1)^{q/p^-}$ for
every admissible $\rho>0$, and the divided-exponent bound is strictly
sharper.

The two upper bounds separate most clearly in the small-mass regime. As
$\rho\downarrow0$, equivalently $K\downarrow0$, we see that $(K+1)^{q/p^-}\to1$, so the older estimate retains the upper bound factor $\displaystyle\tfrac{1}{\eta_0}(\beta-\alpha)^{-q/p^-}$.  On the other hand, the multiplicative term $M_{p/q}(\rho)$ instead satisfies $M_{p/q}(\rho)=K^{q/p^+}\to0$, and so, the new estimate decays to $0$ with the nonlocal mass.
\end{remark}

Our final two lemmata of this section prove additional relationships between the Luxemburg functionals $\Lambda_{\frac{p(\cdot)}{q}}$, $\Lambda_{qp(\cdot)}$, and $\Lambda_{p(\cdot)}$ in case $q>1$.

\begin{lemma}\label{lem:Lpq_le_Lp_concave}
Assume \textnormal{(H2)} and let $q>1$. Then, for every $u\in\mathscr{C}([0,1])$,
\begin{equation*}
 \Lambda_{\frac{p(\cdot)}{q}}(u) \le \Lambda_{p(\cdot)}(u).
\end{equation*}
\end{lemma}

\begin{proof}
The proof is essentially the same as the proof of \cite[Lemma~2.15]{goodrich2025luxemburgnormlocalisationnonlocal},
with the following modifications, but we include the proof for completeness.

So, if $u\equiv\bm{0}$, then the result is immediate. Otherwise, set
\[
 r(t):=\frac{p(t)}{q}, \qquad C:=\Lambda_{r(\cdot)}(u), \qquad v(t):=\frac{|u(t)|}{C}.
\]
By the Luxemburg normalisation identity,
\[
 \int_0^1 v(t)^{r(t)}\,dt = 1.
\]
Since $q>1$, Jensen's inequality, applied to the convex map $x\mapsto x^q$ on $[0,+\infty)$ over the unit-measure interval $[0,1]$, gives
\[
 \int_0^1 v(t)^{p(t)}\,dt
   = \int_0^1 \left(v(t)^{r(t)}\right)^{q}\,dt
   \ge \left( \int_0^1 v(t)^{r(t)}\,dt \right)^{q}
   = 1.
\]
Thus,
\[
 \int_0^1 \left( \frac{|u(t)|}{C} \right)^{p(t)}\,dt \ge 1.
\]
The map
\[
 \delta\mapsto \int_0^1 \left( \frac{|u(t)|}{\delta} \right)^{p(t)}\,dt
\]
is strictly decreasing for $u\not\equiv\bm{0}$. Since
\[
 \int_0^1 \left( \frac{|u(t)|}{\Lambda_{p(\cdot)}(u)} \right)^{p(t)}\,dt = 1,
\]
we conclude that $C\le \Lambda_{p(\cdot)}(u)$. Because
$C=\Lambda_{r(\cdot)}(u)=\Lambda_{p(\cdot)/q}(u)$, the result follows.
\end{proof}

\begin{lemma}\label{lem:gauge-monotone-qp}
Assume \textnormal{(H2)} and let $q>1$. Then, for every
$u\in\mathscr{C}([0,1])$,
\[
 \Lambda_{p(\cdot)}(u)\le\Lambda_{q\,p(\cdot)}(u).
\]
\end{lemma}

\begin{proof}
If $u\equiv\bm{0}$, then both sides vanish. Otherwise set
$\Lambda_p:=\Lambda_{p(\cdot)}(u)>0$ and
\[
 F(t):=\left(\frac{|u(t)|}{\Lambda_p}\right)^{p(t)},
\]
so that $F(t)\ge0$ and, by
Lemma~\ref{lem:Luxemburg_basic_concave}, $\displaystyle\int_0^1 F(t)\,dt=1$. Since $q>1$,
Jensen's inequality applied to the convex map $x\mapsto x^q$ on the
unit-measure interval $[0,1]$ gives
\[
 \int_0^1\left(\frac{|u(t)|}{\Lambda_p}\right)^{q\,p(t)}dt
 =\int_0^1 F(t)^q\,dt
 \ge\left(\int_0^1 F(t)\,dt\right)^{q}=1.
\]
By Lemma~\ref{lem:Luxemburg_basic_concave},
applied with the exponent $s(\cdot)=q\,p(\cdot)$, the map
$\displaystyle\delta\mapsto\int_0^1(|u|/\delta)^{q\,p(t)}\,dt$ is strictly decreasing with
value $1$ at $\delta=\Lambda_{q\,p(\cdot)}(u)$. Since its value at
$\delta=\Lambda_p$ is at least $1$, it follows that
$\Lambda_p\le\Lambda_{q\,p(\cdot)}(u)$.
\end{proof}

\begin{remark}
Lemmata~\ref{lem:Lpq_le_Lp_concave} and
\ref{lem:gauge-monotone-qp} together give
\[
\Lambda_{\frac{p(\cdot)}{q}}(u)
\le
\Lambda_{p(\cdot)}(u)
\le
\Lambda_{q\,p(\cdot)}(u),
\qquad q>1,
\]
thereby recording the relative positions of the three gauges used in the existence
framework of the next section.
\end{remark}

\section{The Auxiliary Operator, Fixed Point Framework, and Existence Theory}

Because we will study \eqref{eq1.1} via a fixed point approach, we
use the same auxiliary-operator construction as in
\cite[(2.8)]{goodrich2025luxemburgnormlocalisationnonlocal}.  We recall it here for the sake of convenience.

Define the
original annular operator
\[
T:\overline{\widehat{V}_{\rho_2}}\setminus\widehat{V}_{\rho_1}
\rightarrow \mathscr{C}([0,1])
\]
by
\begin{equation}\label{eq:T-original}
(Tu)(t):=
\lambda\int_0^1
\left(A\left(\left(b*u^{p(\cdot)}\right)(1)\right)\right)^{-1}
G(t,s)f(s,u(s))\,ds.
\end{equation}
Writing
\[
\Phi(u):=\left(b*u^{p(\cdot)}\right)(1),
\]
we have
\[
\rho_1\le \Phi(u)\le\rho_2
\qquad
\text{for }u\in
\overline{\widehat{V}_{\rho_2}}\setminus\widehat{V}_{\rho_1}.
\]
Thus, \(\Phi(u)\in[\rho_1,\rho_2]\) on
\(\overline{\widehat{V}_{\rho_2}}\setminus\widehat{V}_{\rho_1}\) so that
\(A(\Phi(u))>0\) by \textnormal{(H1)}(3), and the coefficient
\(\big(A(\Phi(u))\big)^{-1}\) in \eqref{eq:T-original} is well defined.
Since the fixed point index requires an operator defined on all of
\(\overline{\widehat{V}_{\rho_2}}\), rather than only on the annulus
\(\overline{\widehat{V}_{\rho_2}}\setminus\widehat{V}_{\rho_1}\) where \(T\)
is available, we extend \(T\) by truncating the nonlocal quantity from below. For
\(u\in\overline{\widehat{V}_{\rho_2}}\), define
\[
\Phi_{\rho_1}(u):=\max\{\Phi(u),\rho_1\}.
\]
Since \(\Phi(u)\le\rho_2\) on \(\overline{\widehat{V}_{\rho_2}}\), it follows
that
\[
\Phi_{\rho_1}(u)\in[\rho_1,\rho_2]
\qquad
\text{for every }
u\in\overline{\widehat{V}_{\rho_2}}.
\]
Hence, \(A(\Phi_{\rho_1}(u))>0\) by \textnormal{(H1)}(3).

We now define the auxiliary operator $S:\overline{\widehat{V}_{\rho_2}}\rightarrow \mathscr{C}([0,1])$ by
\begin{equation}\label{eq:S-auxiliary}
(Su)(t):=
\lambda\int_0^1
\left(A\left(\Phi_{\rho_1}(u)\right)\right)^{-1}
G(t,s)f(s,u(s))\,ds.
\end{equation}
As in the cited construction, \(S\equiv T\) on
\(\overline{\widehat{V}_{\rho_2}}\setminus\widehat{V}_{\rho_1}\), since there we have that
\(\Phi(u)\ge\rho_1\), and so, \(\Phi_{\rho_1}(u)=\Phi(u)\). Consequently, any
fixed point of \(S\) lying in
\(\widehat{V}_{\rho_2}\setminus\overline{\widehat{V}_{\rho_1}}\) is simultaneously a fixed
point of \(T\).
Our next lemma records the compactness and cone-invariance properties needed
for the fixed point index argument.

\begin{lemma}\label{lem:S-compact}
Assume \textnormal{(H1)}, \textnormal{(H2)}, and \textnormal{(H3)}. Then the
auxiliary operator
\(S:\overline{\widehat{V}_{\rho_2}}\to\mathscr{C}([0,1])\) defined in
\eqref{eq:S-auxiliary} is completely continuous, satisfies
\(S(\overline{\widehat{V}_{\rho_2}})\subseteq\mathscr{K}_{\mathrm{hybrid}}\),
and coincides with \(T\) on
\(\overline{\widehat{V}_{\rho_2}}\setminus\widehat{V}_{\rho_1}\). In
particular, \(S\equiv T\) on
\(\partial\widehat{V}_{\rho_1}\cup\partial\widehat{V}_{\rho_2}\).
\end{lemma}

\begin{proof}
The proof follows the argument of
\cite[Lemma~2.22]{goodrich2025luxemburgnormlocalisationnonlocal}; relative to
that proof, the present setting differs in only two ingredients. First, the
boundedness of \(\widehat{V}_{\rho_2}\) is supplied here by
Corollary~\ref{cor:Vrho-bounded}. Second, at the cone-invariance step the
Green-function estimates of \textnormal{(H3)} still give \(Su\ge0\) and
\eqref{eq:thickness}, but where the convex argument concludes with
\((\bm{1}*Su)(1)\ge C_0\|Su\|_{L^{p(\cdot)}}\), the concave regime instead
uses the calibrated coercivity
\((\bm{1}*Su)(1)\ge\widetilde C_0\,\Lambda_{p(\cdot)}(Su)\), which holds for
every nonnegative function satisfying \eqref{eq:thickness} by
Theorem~\ref{thm:coercivity-calibration}. Hence
\(Su\in\mathscr{K}_{\mathrm{hybrid}}\). The remaining assertions, including \(S\equiv T\) on the annulus and hence on
\(\partial\widehat{V}_{\rho_1}\cup\partial\widehat{V}_{\rho_2}\), are
unchanged from the cited proof.
\end{proof}

Finally, we state the topological fixed point theorem
\cite[Lemma~2.1]{cianciaruso1} used below, in which the map (here written
\(\mathcal{T}\) to avoid conflict with the operator \(T\) of
\eqref{eq:T-original}) denotes an arbitrary compact self-map of the cone.

\begin{lemma}\label{lem:fixed-point-index}
Let $U$ be a bounded open set and, with $\mathscr{K}$ a cone in a real Banach
space $\mathscr{X}$, suppose both that
$U_{\mathscr{K}}:=U\cap\mathscr{K}\supseteq\{\bm{0}\}$ and that
$\overline{U_{\mathscr{K}}}\neq\mathscr{K}$. Assume that
$\mathcal{T}:\overline{U_{\mathscr{K}}}\to\mathscr{K}$ is a compact map such
that $x\neq\mathcal{T}x$ for each $x\in\partial U_{\mathscr{K}}$. Then the
fixed point index $i_{\mathscr{K}}\!\left(\mathcal{T},U_{\mathscr{K}}\right)$
has the following properties.
\begin{enumerate}
\item If there exists $v\in\mathscr{K}\setminus\{\bm{0}\}$ such that
$x\neq\mathcal{T}x+\lambda v$ for each $x\in\partial U_{\mathscr{K}}$ and each
$\lambda>0$, then $i_{\mathscr{K}}\!\left(\mathcal{T},U_{\mathscr{K}}\right)=0$.
\item If $\mu x\neq\mathcal{T}x$ for each $x\in\partial U_{\mathscr{K}}$ and
for each $\mu\ge1$, then
$i_{\mathscr{K}}\!\left(\mathcal{T},U_{\mathscr{K}}\right)=1$.
\item If $i_{\mathscr{K}}\!\left(\mathcal{T},U_{\mathscr{K}}\right)\neq0$, then
$\mathcal{T}$ has a fixed point in $U_{\mathscr{K}}$.
\item Let $U^1$ be open in $\mathscr{X}$ with
$\overline{U_{\mathscr{K}}^1}\subseteq U_{\mathscr{K}}$. If
$i_{\mathscr{K}}\!\left(\mathcal{T},U_{\mathscr{K}}\right)=1$ and
$i_{\mathscr{K}}\!\left(\mathcal{T},U_{\mathscr{K}}^{1}\right)=0$, then
$\mathcal{T}$ has a fixed point in
$U_{\mathscr{K}}\setminus\overline{U_{\mathscr{K}}^{1}}$. The same result
holds if $i_{\mathscr{K}}\!\left(\mathcal{T},U_{\mathscr{K}}\right)=0$ and
$i_{\mathscr{K}}\!\left(\mathcal{T},U_{\mathscr{K}}^{1}\right)=1$.
\end{enumerate}
\end{lemma}

With the necessary preliminaries in place, we now prepare the notation for the
existence theorem for \eqref{eq1.1}, subject to the boundary data encoded by
the Green's function \(G\). For \(\rho>0\), recall the upper-localisation
quantity
\[
 B_{\infty,\rho}^{(q)}
 := \frac{1}{\eta_0(\beta-\alpha)^{q/p^-}}\,M_{p/q}(\rho),
\]
where \(M_{p/q}(\rho)\) is the divided-exponent upper-bound quantity from
Lemma~\ref{lem:reverse-holder-divided-exponent}. We also make the following declarations for future reference.
\begin{align*}
m_{\rho}
&:= \left[
     \left(\frac{\rho}{(b*\bm{1})(1)}\right)^{\frac{1}{p^+}}
     + \varepsilon_1(\rho,b)
   \right] \\[6pt]
M_{\rho}
&:= B_{\infty,\rho}^{(q)} \\[6pt]
C_2(q)
&:= \left(
     \int_0^1 |b(1-s)|^{\frac{q}{q-1}}\,ds
   \right)^{\frac{q-1}{q}} \\[8pt]
{N_1}^{(q)}
&{:= \frac{\lambda}{A(\rho_1)}\,
  f^{m}_{[\alpha,\beta]\times[\eta_0 m_{\rho_1},\,B_{\infty,\rho_1}^{(q)}]}\,
  \inf_{t\in[\alpha,\beta]}
  \int_{\alpha}^{\beta} G(t,s)\,ds} \\[8pt]
{Y_1}^{(q)}
&{:= N_1^{(q)}(\beta-\alpha)^{\frac{q}{p^-}}}
\end{align*}
Note that \(\varepsilon_1(\rho,b)\), which occurs in the definition of $m_{\rho}$, is the piecewise quantity defined in
Lemma~\ref{lem:Lp_lower_combined}. The quantities \(m_\rho\), \(M_\rho\), \(N_1^{(q)}\), \(Y_1^{(q)}\) and
\(C_2(q)\) will be used in the existence theorem below.

We now state and prove our existence theorem.  Since in Section 2 we have established a functional analytic framework analogous to the convex setting studied in \cite{goodrich2025luxemburgnormlocalisationnonlocal}, it is unsurprising that the proof of Theorem \ref{thm:existence-concave} does not differ much from the analogous result \cite[Theorem 2.25]{goodrich2025luxemburgnormlocalisationnonlocal} in the convex setting.  Nonetheless, we largely include the proof of Theorem \ref{thm:existence-concave} both for completeness and to highlight the small discrepancies that do arise.

Throughout the following theorem and its proof, the auxiliary exponents
$p(\cdot)/q$ and $q\,p(\cdot)$ occur alongside the base exponent $p(\cdot)$.
The cone-restricted quasi-norm notation of
Theorem~\ref{thm:cone-promotion}(iii) is defined only at the base exponent on
$\mathscr K_{\mathrm{hybrid}}$, and so, we write every modular-induced scaling
functional as a gauge $\Lambda_{s(\cdot)}$ during the index computations,
reverting to $\|\cdot\|_{L^{p(\cdot)}}=\Lambda_{p(\cdot)}(\cdot)$ only for the
final localisation of $u_0$.

\begin{theorem}\label{thm:existence-concave}
Suppose that \textnormal{(H1)}, \textnormal{(H2)}, and \textnormal{(H3)}
hold, and that
\[
 b^{\frac{1}{1-q}}\in L^1((0,1])
 \qquad\text{and}\qquad
 b\in L^{\frac{q}{q-1}}((0,1])
\]
for some \(q>1\). In addition, suppose that the numbers \(\rho_1\) and
\(\rho_2\) from \textnormal{(H1)(3)} are such that each of the following
inequalities holds.
\vskip0.2cm
\begin{enumerate}
\item
\begin{equation}\label{eq:theorem_condition1_concave}
{
\begin{cases}
C_1(q)\,\left(Y_1^{(q)}\right)^{p^-}>\rho_1, & \text{if } \dfrac{\rho_1}{C_1(q)}\ge 1, \\[6pt]
C_1(q)\,\left(Y_1^{(q)}\right)^{p^+}>\rho_1, & \text{if } 0<\dfrac{\rho_1}{C_1(q)}<1.
\end{cases}
}
\end{equation}
\vskip0.2cm
\item
\begin{equation}\label{eq:theorem_condition2_concave}
{
\begin{cases}
C_2(q)\left(
 \left[\dfrac{\lambda\,f^{M}_{[0,1]\times[0,B_{\infty,\rho_2}^{(q)}]}}{A(\rho_2)}\right]
 \Lambda_{q\,p(\cdot)}\!\left(\int_0^1 G(\cdot,s)\,ds\right)
\right)^{p^+} <\rho_2,
& \text{if } \dfrac{\rho_2}{C_2(q)}\ge 1, \\[10pt]
C_2(q)\left(
 \left[\dfrac{\lambda\,f^{M}_{[0,1]\times[0,B_{\infty,\rho_2}^{(q)}]}}{A(\rho_2)}\right]
 \Lambda_{q\,p(\cdot)}\!\left(\int_0^1 G(\cdot,s)\,ds\right)
\right)^{p^-} <\rho_2,
& \text{if } 0<\dfrac{\rho_2}{C_2(q)}<1.
\end{cases}
}
\end{equation}
\end{enumerate}
\vskip0.2cm
Then problem \eqref{eq1.1}, subject to the boundary data generated by the
Green's function \(G\), has at least one positive solution
\[
 u_0\in\widehat V_{\rho_2}\setminus\overline{\widehat V_{\rho_1}}.
\]
Moreover, \(u_0\) satisfies the Luxemburg gauge localisation
\begin{equation}
 \eta_0(\beta-\alpha)^{\frac{1}{p^-}}m_{\rho_1}
 \le \|u_0\|_{L^{p(\cdot)}}
 \le M_{\rho_2}.
\end{equation}
\end{theorem}

\begin{proof}
As mentioned before the statement of the theorem, because of the machinery built in Section 2, the proof largely follows that of \cite[Theorem 2.25]{goodrich2025luxemburgnormlocalisationnonlocal} in the convex setting.  So, in places where this overlap is complete, we will omit the details and refer to \cite[Theorem 2.25]{goodrich2025luxemburgnormlocalisationnonlocal}. 

Our first task is to invoke part~(1) of Lemma~\ref{lem:fixed-point-index}.  Suppose for contradiction that there exist \(u\in\partial\widehat V_{\rho_1}\)
and \(\mu>0\) such that
\[
 u=Su+\mu\bm 1.
\]
Since \(u\in\partial\widehat V_{\rho_1}\),
Lemma~\ref{lem:boundary-characterisation-concave} gives
\[
 \Phi(u)
 =\left(b*u^{p(\cdot)}\right)(1)
 =\int_0^1 b(1-t)u(t)^{p(t)}\,dt
 =\rho_1.
\]
Moreover, for each \(t\in[0,1]\),
\[
 u(t)=(Su)(t)+\mu.
\]
Since \(b\ge0\) a.e., we obtain
\begin{equation}\label{eq:inner-power-comparison}
 b(1-t)u(t)^{p(t)} \ge b(1-t)(Su)(t)^{p(t)}
 \qquad\text{for a.e. }t\in[0,1].
\end{equation}
Integrating \eqref{eq:inner-power-comparison} over \([0,1]\) gives
\begin{equation}\label{eq:inner-functional-comparison}
\rho_1=\left(b*u^{p(\cdot)}\right)(1) \ge \left(b*(Su)^{p(\cdot)}\right)(1),
\end{equation}
and so,
\begin{equation}\label{eq:rho1_ge_Su_functional}
 \rho_1 \ge \left(b*(Su)^{p(\cdot)}\right)(1).
\end{equation}
By the reverse H\"older inequality,
\begin{equation}\label{eq:inner-reverse-holder}
\begin{split}
\left(b*(Su)^{p(\cdot)}\right)(1)
&= \int_0^1 b(1-t)(Su)(t)^{p(t)}\,dt\\
&= \int_0^1 b(1-t)\left((Su)(t)^{p(t)/q}\right)^q\,dt\\
&\ge \left(\int_0^1 b(1-t)^{\frac{1}{1-q}}\,dt\right)^{1-q}
    \left(\int_0^1 (Su)(t)^{p(t)/q}\,dt\right)^q\\
&= C_1(q)\left(I_{\frac{p(\cdot)}{q}}(Su)\right)^q.
\end{split}
\end{equation}
Combining \eqref{eq:rho1_ge_Su_functional} with
\eqref{eq:inner-reverse-holder}, we obtain
\begin{equation}\label{eq:rho1_ge_C1_I_r}
 \rho_1 \ge C_1(q)\left(I_{\frac{p(\cdot)}{q}}(Su)\right)^q.
\end{equation}

We now show that \eqref{eq:rho1_ge_C1_I_r}, together with the hypotheses in the
statement of the theorem, leads to a contradiction. To begin, recall the
notation introduced prior to the statement of the theorem. Since \(u\in\partial\widehat V_{\rho_1}\), \cite[Lemma~3.3]{goodrich2025unified} and Lemma~\ref{lem:upper-localisation-divided-exponent} give
\[
 \|u\|_\infty\ge m_{\rho_1}
 \qquad\text{and}\qquad
 \|u\|_\infty\le B_{\infty,\rho_1}^{(q)}.
\]
Moreover, since \(u\in\mathscr K_{\mathrm{hybrid}}\), inequality \eqref{eq:thickness}
implies
\[
 u(s)\ge \eta_0\|u\|_\infty\ge \eta_0 m_{\rho_1}
 \qquad\text{for all }s\in[\alpha,\beta].
\]
Consequently,
\[
 u(s)\in \left[\eta_0 m_{\rho_1},\, B_{\infty,\rho_1}^{(q)}\right]
 \qquad\text{for all }s\in[\alpha,\beta].
\]
Therefore, for $u\in\partial\widehat{V}_{\rho_1}$ and
\(t\in[\alpha,\beta]\), we have
\begin{equation}\label{eq:Su-min-integral-concave}
\begin{split}
(Su)(t)
&= \frac{\lambda}{A(\rho_1)} \int_0^1 G(t,s)f(s,u(s))\,ds\\
&\ge \frac{\lambda}{A(\rho_1)}\,
    f^m_{[\alpha,\beta]\times[\eta_0 m_{\rho_1},\,B_{\infty,\rho_1}^{(q)}]}\,
    \inf_{\tau\in[\alpha,\beta]} \int_\alpha^\beta G(\tau,s)\,ds= N_1^{(q)}.
\end{split}
\end{equation}
Note that in \eqref{eq:Su-min-integral-concave} we have used the fact that, for
each \(s\in[\alpha,\beta]\),
\[
 \eta_0 m_{\rho_1}
 =\eta_0\left[
    \left(\frac{\rho_1}{(b*\bm 1)(1)}\right)^{\frac{1}{p^+}}
    + \varepsilon_1(\rho_1,b)
  \right]
 \le \eta_0\|u\|_\infty
 \le u(s)
 \le \|u\|_\infty
 \le B_{\infty,\rho_1}^{(q)}.
\]

We next use \eqref{eq:Su-min-integral-concave} to obtain a lower bound for the
Luxemburg gauge of \(Su\) at the divided exponent. Put
\[
 r(t):=\frac{p(t)}{q}.
\]
By \eqref{eq:Su-min-integral-concave}, we have
\[
 (Su)(t)\ge N_1^{(q)} \qquad\text{for all }t\in[\alpha,\beta].
\]
We claim that
\[
 \Lambda_{r(\cdot)}(Su) \ge N_1^{(q)}(\beta-\alpha)^{1/r^-}.
\]
Then repeating exactly \cite[pp. 22--23]{goodrich2025luxemburgnormlocalisationnonlocal}, we suppose that $0<k<N_1^{(q)}(\beta-\alpha)^{1/r^-}$, and are led to the conclusion that no such \(k\) can belong to the admissible set
\[
 \left\{ k>0: \int_0^1 \left(\frac{|(Su)(t)|}{k}\right)^{r(t)}\,dt \le 1 \right\}.
\]
Taking the infimum over this admissible set gives
\begin{equation}\label{eq:lower-gauge-bound-divided}
 \Lambda_{r(\cdot)}(Su)
 =\Lambda_{\frac{p(\cdot)}{q}}(Su)
 \ge N_1^{(q)}(\beta-\alpha)^{1/r^-}
 = N_1^{(q)}(\beta-\alpha)^{q/p^-}
 = Y_1^{(q)}.
\end{equation}

We now distinguish between two different regimes for the ratio $\rho_1/C_1(q)$,
since the two cases of Corollary~\ref{cor:mod_scaling_concave} apply depending
on whether $\Lambda_{\frac{p(\cdot)}{q}}(Su)$ lies above or below $1$.
Specifically, recall from \eqref{eq:theorem_condition1_concave} and
\eqref{eq:rho1_ge_C1_I_r} that
\[
C_1(q)\left(I_{\frac{p(\cdot)}{q}}(Su)\right)^q \le \rho_1 <
C_1(q)\bigl(Y_1^{(q)}\bigr)^{p^{\pm}},
\]
and so, the comparison between $\rho_1$ and $C_1(q)$ determines whether the
modular is evaluated in the ``large-gauge'' or ``small-gauge'' regime.  We want to show that in each case a contradiction thus follows; this part of the proof tracks \cite[pp. 23--24]{goodrich2025luxemburgnormlocalisationnonlocal} very closely.

\subsubsection*{Case 1: $\rho_1/C_1(q)\ge 1$}
By hypothesis we have that
$C_1(q)\bigl(Y_1^{(q)}\bigr)^{p^-}>\rho_1$ so that $Y_1^{(q)}>1$. From
\eqref{eq:lower-gauge-bound-divided} we see that
$\Lambda_{\frac{p(\cdot)}{q}}(Su)\ge Y_1^{(q)}$, and so,
$\Lambda_{\frac{p(\cdot)}{q}}(Su)>1$. Corollary~\ref{cor:mod_scaling_concave} implies that
\[
I_{\frac{p(\cdot)}{q}}(Su)
 \ge
 \Lambda_{\frac{p(\cdot)}{q}}(Su)^{\frac{p^-}{q}}
 \ge
 \bigl(Y_1^{(q)}\bigr)^{\frac{p^-}{q}}
\]
so that by both \eqref{eq:theorem_condition1_concave} and \eqref{eq:rho1_ge_C1_I_r} we are led to
\begin{equation}\label{eq:rho1-upper-contradiction}
C_1(q)\bigl(Y_1^{(q)}\bigr)^{p^-} > \rho_1
 \ge C_1(q)\left(I_{\frac{p(\cdot)}{q}}(Su)\right)^q
\;\implies\;
\bigl(Y_1^{(q)}\bigr)^{p^-} > \left(I_{\frac{p(\cdot)}{q}}(Su)\right)^q.
\end{equation}
But \eqref{eq:rho1-upper-contradiction} yields the contradiction
\begin{equation*}
\bigl(Y_1^{(q)}\bigr)^{p^-}
 > \left(I_{\frac{p(\cdot)}{q}}(Su)\right)^q
 \ge \bigl(Y_1^{(q)}\bigr)^{\,q\cdot\frac{p^-}{q}}
 = \bigl(Y_1^{(q)}\bigr)^{p^-}.
\end{equation*}

\subsubsection*{Case 2: $0 < \rho_1/C_1(q) < 1$}
Inequalities \eqref{eq:theorem_condition1_concave} and \eqref{eq:rho1_ge_C1_I_r} collectively imply that
\[
C_1(q)\bigl(Y_1^{(q)}\bigr)^{p^+} > \rho_1
 \ge C_1(q)\left(I_{\frac{p(\cdot)}{q}}(Su)\right)^q
\;\implies\;
\bigl(Y_1^{(q)}\bigr)^{p^+} > \left(I_{\frac{p(\cdot)}{q}}(Su)\right)^q.
\]
Note, moreover, that
$I_{\frac{p(\cdot)}{q}}(Su) < 1$ so that $\Lambda_{\frac{p(\cdot)}{q}}(Su) < 1$. Corollary~\ref{cor:mod_scaling_concave}, together with the inequality
$\Lambda_{\frac{p(\cdot)}{q}}(Su) \ge Y_1^{(q)}$ from
\eqref{eq:lower-gauge-bound-divided}, yields
\[
\bigl(Y_1^{(q)}\bigr)^{p^+}
 > \left(I_{\frac{p(\cdot)}{q}}(Su)\right)^q
 \ge \bigl(Y_1^{(q)}\bigr)^{\,q\cdot\frac{p^+}{q}}
 = \bigl(Y_1^{(q)}\bigr)^{p^+},
\]
which again yields the desired contradiction.

Since both cases lead to a contradiction, it follows that $u\neq Su+\mu\bm{1}$ whenever
$u\in\partial\widehat V_{\rho_1}$. Therefore, from part~(1)
of Lemma~\ref{lem:fixed-point-index} we conclude that
\begin{equation}\label{eq:index-zero}
i_{\mathscr{K}}\!\left(S,\widehat{V}_{\rho_1}\right)=0.
\end{equation}

We next appeal to part~(2) of Lemma~\ref{lem:fixed-point-index}. To this end, we
show that, whenever \(u\in\partial\widehat V_{\rho_2}\) and \(\mu\ge1\), one has
\[
 Su\ne \mu u.
\]
Suppose for contradiction that there exist
\(u\in\partial\widehat V_{\rho_2}\) and \(\mu\ge1\) such that $Su\equiv \mu u$, which leads to
\begin{equation}\label{eq:outer-power-identity}
 (Su)(t)^{p(t)}
 =\bigl(\mu u(t)\bigr)^{p(t)}
 =\mu^{p(t)}u(t)^{p(t)}
 \qquad\text{for all }t\in[0,1].
\end{equation}
Since \(u\in\partial\widehat V_{\rho_2}\),
Lemma~\ref{lem:boundary-characterisation-concave} gives \(\Phi(u)=\rho_2\), and because \(\rho_2>\rho_1\), we have \(\Phi_{\rho_1}(u)=\max\{\Phi(u),\rho_1\}=\rho_2\),
and so, \(A(\Phi_{\rho_1}(u))=A(\rho_2)\).
Similar to \cite[p. 24]{goodrich2025luxemburgnormlocalisationnonlocal}, set
\[
 g(s):=\frac{f(s,u(s))}{\,f^{M}_{[0,1]\times[0,B_{\infty,\rho_2}^{(q)}]}\,}\in[0,1],
\]
where \(B_{\infty,\rho_2}^{(q)}\) is from
Lemma~\ref{lem:upper-localisation-divided-exponent}. Then
\[
 (Su)(t)
 = \frac{\lambda}{A(\rho_2)}\int_0^1 G(t,s)f(s,u(s))\,ds
 = \left[\frac{\lambda f^{M}_{[0,1]\times[0,B_{\infty,\rho_2}^{(q)}]}}{A(\rho_2)}\right]
   \int_0^1 G(t,s)g(s)\,ds.
\]

Since \(0\le g(s)\le 1\) and \(G(t,s)\ge0\), we obtain
\[
 0 \le \frac{(Su)(t)}
      {\left[\dfrac{\lambda\,f^{M}_{[0,1]\times[0,B_{\infty,\rho_2}^{(q)}]}}{A(\rho_2)}\right]}
  = \int_0^1 G(t,s)g(s)\,ds
  \le \int_0^1 G(t,s)\,ds.
\]
By Lemma~\ref{lem:gauge-lattice-monotone} applied with
$s(\cdot)=q\,p(\cdot)$, which is admissible since the lemma requires only $s>0$, we
obtain
\begin{equation}\label{eq:Su-Lqp-upper-concave}
 \Lambda_{q\,p(\cdot)}(Su)
 \le
 \left[\frac{\lambda f^{M}_{[0,1]\times[0,B_{\infty,\rho_2}^{(q)}]}}{A(\rho_2)}\right]\,
 \Lambda_{q\,p(\cdot)}\!\left(\int_0^1 G(\cdot,s)\,ds\right).
\end{equation}
Using \(\left(b*u^{p(\cdot)}\right)(1)=\rho_2\), the bound
\(\mu^{p(t)}\ge1\) (valid since \(\mu\ge1\) and \(p(t)>0\)), the identity
\eqref{eq:outer-power-identity}, and the forward H\"older inequality with
conjugate exponents \(q\) and \(\displaystyle\frac{q}{q-1}\), analogous to \cite[(2.23)]{goodrich2025luxemburgnormlocalisationnonlocal} we obtain
\begin{equation}\label{eq:rho2-upper-bound-concave}
\begin{split}
\rho_2
&\le\int_0^1 b(1-t)\,\mu^{p(t)}\big(u(t)\big)^{p(t)}\,dt\\
&=\int_0^1 b(1-t)\big((Su)(t)\big)^{p(t)}\,dt\\
&\le\left(\int_0^1 \big|b(1-t)\big|^{\frac{q}{q-1}}\,dt\right)^{\!\frac{q-1}{q}}
  \left(\int_0^1 \big|(Su)(t)\big|^{q\,p(t)}\,dt\right)^{1/q}=C_2(q)\left(I_{q\,p(\cdot)}(Su)\right)^{\frac{1}{q}}.
\end{split}
\end{equation}

It remains to show that \eqref{eq:Su-Lqp-upper-concave} and \eqref{eq:rho2-upper-bound-concave} are jointly contradictory.  However, by means of Corollary~\ref{cor:mod_scaling_concave}, we may proceed exactly as in \cite[pp. 25--26]{goodrich2025luxemburgnormlocalisationnonlocal}, replacing in the cited reference $B_{\infty,\rho_2}$ with $B_{\infty,\rho_2}^{(q)}$ and, likewise, $\Vert\cdot\Vert_{L^{q\cdot p(\cdot)}}$ with $\Lambda_{qp(\cdot)}$.  Doing so, we conclude that no $u\in\partial\widehat V_{\rho_2}$
satisfies $Su\equiv\mu u$ with $\mu\ge1$ in this case.  Consequently, part~(2) of Lemma~\ref{lem:fixed-point-index} implies that
\begin{equation}\label{eq:index-one}
i_{\mathscr{K}}\!\left(S,\widehat V_{\rho_2}\right)=1.
\end{equation}

In summary, combining \eqref{eq:index-zero} and \eqref{eq:index-one}, we infer
from part~(4) of Lemma~\ref{lem:fixed-point-index} that there exists
\[
 u_0\in \widehat V_{\rho_2}\setminus\overline{\widehat V_{\rho_1}}
\]
such that $Su_0=u_0$, from which it follows that $Tu_0\equiv u_0$, and so, $u_0$ is a positive solution of
\eqref{eq1.1}, subject to the boundary data generated by the Green's function
$G$. Since $u_0\notin\overline{\widehat V_{\rho_1}}$, we have
$\Phi(u_0)>\rho_1>0$.  The map $\tau\mapsto\Phi(\tau u_0)$ is continuous by
Lemma~\ref{lem:Phi-continuity-concave} and strictly increasing by
Lemma~\ref{lem:positive-scaling-admissible}, and satisfies
$\Phi(\tau u_0)\le\tau^{p^-}\Phi(u_0)\to0$ as $\tau\downarrow0$; hence, there
exists $\tau_0\in(0,1)$ with $\Phi(\tau_0u_0)=\rho_1$, and
Lemmata~\ref{lem:positive-scaling-admissible}
and~\ref{lem:boundary-characterisation-concave} give
$\tau_0u_0\in\partial\widehat V_{\rho_1}$.  Applying
Lemma~\ref{lem:Lp_lower_combined} to $\tau_0u_0$ and using the positive
homogeneity of Lemma~\ref{lem:gauge-lattice-monotone} together with
$\tau_0<1$ gives the lower bound.  Since
$u_0\in\widehat V_{\rho_2}$, Corollary~\ref{cor:Vrho-bounded} gives $\displaystyle\|u_0\|_{L^{p(\cdot)}}\le
B_{\infty,\rho_2}^{(q)}=M_{\rho_2}$.  Consequently, we obtain the Luxemburg gauge localisation
\begin{equation}
 \eta_0(\beta-\alpha)^{\frac{1}{p^-}}m_{\rho_1}
 \le \|u_0\|_{L^{p(\cdot)}}
 \le M_{\rho_2},
\end{equation}
as desired.
\end{proof}

Before turning to our example, we record a direct \textquotedblleft$q$-free criterion\textquotedblright\ available
when $b\equiv\bm 1$.  Although this choice of $b$ satisfies the hypotheses of
Theorem~\ref{thm:existence-concave} for every admissible $q>1$, the criterion
below is obtained directly from the identity
\[
\Phi(u)=\left(\bm 1*u^{p(\cdot)}\right)(1)=I_{p(\cdot)}(u),
\]
rather than by setting $q=1$ in Theorem~\ref{thm:existence-concave}.  We state only the parameter
regime required in Example~\ref{example2.12}; the complementary regimes are
obtained analogously by using the corresponding modular--gauge exponents.

\begin{corollary}\label{cor:b-equiv-one}
Assume \textnormal{(H1)}, \textnormal{(H2)}, and \textnormal{(H3)}, and suppose
that $b\equiv\bm 1$.  Let the numbers $\rho_1$ and $\rho_2$ from
\textnormal{(H1)}(3) satisfy $0<\rho_1<1\le\rho_2$.  Define
\[
B_{\infty,\rho_1}^{*}:=\frac{\rho_1^{1/p^+}}{\eta_0(\beta-\alpha)^{1/p^-}},
\qquad
B_{\infty,\rho_2}^{*}:=\frac{\rho_2^{1/p^-}}{\eta_0(\beta-\alpha)^{1/p^-}},
\]
and put
\[
N_1^{*}:=\frac{\lambda}{A(\rho_1)}
f^{m}_{[\alpha,\beta]\times[\eta_0\rho_1^{1/p^-},\,B_{\infty,\rho_1}^{*}]}
\inf_{t\in[\alpha,\beta]}\int_\alpha^\beta G(t,s)\,ds,
\qquad h(t):=\int_0^1G(t,s)\,ds.
\]
Suppose that $0<N_1^{*}\le1$ and that
\begin{equation}\label{eq:cor-inner}
(\beta-\alpha)\bigl(N_1^{*}\bigr)^{p^+}>\rho_1
\end{equation}
and
\begin{equation}\label{eq:cor-outer}
\left(\frac{\lambda f^{M}_{[0,1]\times[0,B_{\infty,\rho_2}^{*}]}}{A(\rho_2)}
\Lambda_{p(\cdot)}(h)\right)^{p^+}<\rho_2.
\end{equation}
Then problem \eqref{eq1.1}, subject to the boundary data generated by the
Green's function $G$, has at least one positive solution
$u_0\in\widehat V_{\rho_2}\setminus\overline{\widehat V_{\rho_1}}$.
Moreover, $u_0$ satisfies the Luxemburg gauge localisation
\[
\eta_0(\beta-\alpha)^{1/p^-}\rho_1^{1/p^-}\le\|u_0\|_{L^{p(\cdot)}}\le\rho_2^{1/p^-}.
\]
\end{corollary}

\begin{proof}
Since $b\equiv\bm 1$, we have $\Phi\equiv I_{p(\cdot)}$, so that
$I_{p(\cdot)}(u)\le\rho$ on $\overline{\widehat V_\rho}$, with equality on
$\partial\widehat V_\rho$ by Lemma~\ref{lem:boundary-characterisation-concave}.
The proof of Theorem~\ref{thm:existence-concave} applies with only the following
three modifications, each of which replaces a step that passed through the
divided exponent.

\emph{(i) The upper bounds.}  If $u\in\overline{\widehat V_{\rho_1}}$
and $u\equiv\bm 0$, then the desired estimate is immediate.  If
$u\not\equiv\bm 0$, then $I_{p(\cdot)}(u)\le\rho_1<1$ forces
$0<\Lambda_{p(\cdot)}(u)<1$; otherwise part~(1) of
Corollary~\ref{cor:mod_scaling_concave} would give
$I_{p(\cdot)}(u)\ge1$.  Part~(2) of that corollary therefore gives
\[
  \Lambda_{p(\cdot)}(u)\le\rho_1^{1/p^+}.
\]
If $u\in\overline{\widehat V_{\rho_2}}$, then part~(1) of that corollary gives
\[
  \Lambda_{p(\cdot)}(u)\le\rho_2^{1/p^-}
\]
when $\Lambda_{p(\cdot)}(u)\ge1$, and otherwise $\Lambda_{p(\cdot)}(u)<1\le\rho_2^{1/p^-}$.
Lemma~\ref{lem:thickness-concave} converts these into
\[
  \|u\|_\infty\le B^{*}_{\infty,\rho_i},
\]
which replace $B^{(q)}_{\infty,\rho_i}$ throughout.

\emph{(ii) The inner mass inequality, without H\"older.}  For
$u\in\partial\widehat V_{\rho_1}$ with $u=Su+\mu\bm 1$ we have $u\ge Su$, and so,
\[
  \rho_1=I_{p(\cdot)}(u)\ge I_{p(\cdot)}(Su).
\]
Since $\|u\|_\infty\ge\rho_1^{1/p^-}$ by \cite[Lemma~3.3]{goodrich2025unified}, the
thickness condition gives
\[
  \eta_0\rho_1^{1/p^-}\le u(s)\le B^{*}_{\infty,\rho_1}
\]
on $[\alpha,\beta]$, and so, $(Su)(t)\ge N_1^{*}$ there as in \eqref{eq:Su-min-integral-concave}.  As
$0<N_1^{*}\le1$ and $p(t)\le p^+$, it follows that
\[
  \rho_1\ge I_{p(\cdot)}(Su)\ge\int_\alpha^\beta\bigl((Su)(t)\bigr)^{p(t)}\,dt
  \ge(\beta-\alpha)\bigl(N_1^{*}\bigr)^{p^+}>\rho_1
\]
by \eqref{eq:cor-inner}, which is a contradiction.

\emph{(iii) The outer estimate at the base exponent.}  For
$u\in\partial\widehat V_{\rho_2}$ with $Su=\mu u$, $\mu\ge1$, we have
\[
  \rho_2=I_{p(\cdot)}(u)\le I_{p(\cdot)}(Su),
\]
and $\Phi_{\rho_1}(u)=\rho_2$ gives
\[
  0\le Su\le\left[\frac{\lambda f^{M}_{[0,1]\times[0,B^{*}_{\infty,\rho_2}]}}{A(\rho_2)}\right]h,
\]
from which it follows that
\[
\Lambda_{p(\cdot)}(Su)
\le
\frac{\lambda
f^{M}_{[0,1]\times[0,B^{*}_{\infty,\rho_2}]}}
{A(\rho_2)}
\Lambda_{p(\cdot)}(h)
\]
by Lemma~\ref{lem:gauge-lattice-monotone}; note that here $h$ is the quantity defined in the statement of the corollary.  Since $I_{p(\cdot)}(Su)\ge\rho_2\ge1$
we have $\Lambda_{p(\cdot)}(Su)\ge1$, so part~(1) of
Corollary~\ref{cor:mod_scaling_concave} gives 
\[
  \rho_2 \le \left( \frac{\lambda f^{M}_{[0,1]\times[0,B^{*}_{\infty,\rho_2}]}}{A(\rho_2)} \Lambda_{p(\cdot)}(h) \right)^{p^+} < \rho_2
\]
by \eqref{eq:cor-outer}, a contradiction.

Both the remaining index argument and the existence of
\[
u_0\in\widehat V_{\rho_2}\setminus\overline{\widehat V_{\rho_1}}
\]
follow exactly as in Theorem~\ref{thm:existence-concave}.  To obtain
the lower localisation, choose $\tau_0\in(0,1)$ as in the final part of the
proof of that theorem so that
\[
\Phi(\tau_0u_0)=\rho_1
\qquad\text{and hence}\qquad
\tau_0u_0\in\partial\widehat V_{\rho_1}.
\]
Lemma~\ref{lem:Lp_lower_combined} and the positive homogeneity of the
Luxemburg gauge then give
\[
\tau_0\Lambda_{p(\cdot)}(u_0)
=
\Lambda_{p(\cdot)}(\tau_0u_0)
\ge
\eta_0(\beta-\alpha)^{1/p^-}m_{\rho_1}.
\]
Since $\tau_0<1$, $(b*\bm 1)(1)=1$, and $\rho_1<1$, we have
$m_{\rho_1}=\rho_1^{1/p^-}$, and, therefore,
\[
\Lambda_{p(\cdot)}(u_0)
\ge
\eta_0(\beta-\alpha)^{1/p^-}\rho_1^{1/p^-}.
\]
The upper localisation follows directly from
$u_0\in\widehat V_{\rho_2}$ and the outer cap established in part~(i):
\[
\Lambda_{p(\cdot)}(u_0)\le\rho_2^{1/p^-}.
\]
This completes the proof of the corollary.
\end{proof}

We conclude with an example in order to illustrate not only the application of Corollary \ref{cor:b-equiv-one} but also the improvements that our framework affords over existing methodologies.

\begin{example}\label{example2.12}
We illustrate the effect of our framework on a concave modification of a problem
studied in \cite[Example 2.13]{goodrich2025unified}. That example was set in the
convex regime; here we retain its coefficient $A$, its annulus, and its cone
geometry, but replace the exponent by one taking values in $(0,1)$, so that the
problem falls within the concave regime treated in the present work. Hence, $A$, the
Green's function, and $\alpha,\beta,\eta_0$ are common to both examples.

So, as in \cite[Example 2.13]{goodrich2025unified}, we take $b\equiv\bm{1}$.  Set
\[
 A(t):=\frac{1000}{3}\,t\sin\!\left(\frac{\pi}{6}t\right)
 \qquad\text{and}\qquad
 p(t):=\frac{3}{4}+\frac{1}{8}\cos(\pi t).
\]
With the Dirichlet Green's function $G$ from \eqref{eq:G_dirichlet}, consider the
nonlocal boundary value problem
\begin{equation}\label{eq2.35}
\begin{aligned}
- A\!\left(\int_0^1\big(u(s)\big)^{\frac{3}{4}+\frac{1}{8}\cos(\pi s)}\,ds\right) u''(t)
 &= \lambda f\big(t,u(t)\big), \quad 0<t<1,\\
u(0)&=u(1)=0.
\end{aligned}
\end{equation}

For this exponent function $p$ we have $p^-=\tfrac{5}{8}$ and $p^+=\tfrac{7}{8}$,
so that $0<p^-\le p(t)\le p^+<1$ and \textnormal{(H2)} holds. Choose
$\displaystyle\alpha=\frac{1}{4}$ and $\displaystyle\beta=\frac{3}{4}$; then
(see \cite{erbe_positive_2004})
\[
 \eta_0=\frac{1}{4},\quad
 \max_{\tau\in[0,1]}\int_{\alpha}^{\beta} G(\tau,s)\,ds=\frac{3}{32},\quad
 \inf_{t\in[\alpha,\beta]}\int_{\alpha}^{\beta} G(t,s)\,ds=\frac{1}{16},\quad
 \max_{\tau\in[0,1]}\int_{0}^{1} G(\tau,s)\,ds=\frac{1}{8}.
\]
We fix $\rho_1=\dfrac{1}{2500}$ and $\rho_2=3$. The coefficient $A$ is positive on
$[\rho_1,\rho_2]\subset(0,6)$, since $\sin(\pi t/6)>0$ there; in particular
\[
 A(\rho_1)=A\!\left(\tfrac{1}{2500}\right)\approx 2.7925\times10^{-5},
 \qquad
 A(\rho_2)=A(3)=\frac{1000}{3}\cdot 3\cdot\sin\frac{\pi}{2}=1000.
\]
Since $b\equiv\bm{1}$, both comparisons below are carried out without reference to
the auxiliary parameter $q$, and are therefore governed by
Corollary~\ref{cor:b-equiv-one} rather than by Theorem~\ref{thm:existence-concave}
itself. Indeed, the nonlocal boundary functional reduces
exactly to the modular,
\[
 \left(b*u^{p(\cdot)}\right)(1)=\int_0^1 u(t)^{p(t)}\,dt=I_{p(\cdot)}(u),
\]
so that, on $\partial\widehat V_\rho$, one has $I_{p(\cdot)}(u)=\rho$, and the
upper cap follows directly from the modular--Luxemburg gauge comparison of
Corollary~\ref{cor:mod_scaling_concave}, with no reverse-H\"older step required.

The two competing upper caps share the thickness prefactor
\[
 \frac{1}{\eta_0(\beta-\alpha)^{1/p^-}}
 =\frac{1}{\tfrac14\,(1/2)^{8/5}}
 =4\cdot 2^{8/5}
 \approx 12.1257.
\]
The cap obtained from the older concave sub-additivity argument
\cite[Theorem~3.6]{goodrich2025unified} carries a structural additive correction,
\[
 \overline M_\rho
 :=\frac{1}{\eta_0(\beta-\alpha)^{1/p^-}}\,(\rho+1)^{1/p^-}
 =4\cdot 2^{8/5}\,(\rho+1)^{8/5},
\]
whereas the cap furnished by the present hybrid-cone method is purely
multiplicative, namely,
\[
 B^{*}_{\infty,\rho}
 :=\frac{1}{\eta_0(\beta-\alpha)^{1/p^-}}
 \begin{cases}
   \rho^{1/p^+}, & 0<\rho<1,\\[4pt]
   \rho^{1/p^-}, & \rho\ge1.
 \end{cases}
\]
The entire difference between the two is the additive $+1$ inside the bracket,
which the sub-additivity step \cite[Lemma~3.1]{goodrich2025unified} introduces and
the direct modular--gauge comparison removes.
\subsection*{Original concave sub-additivity framework \cite[Theorem~3.6]{goodrich2025unified}}
Provided that $\lambda$ and $f$ satisfy both
\begin{equation}\label{eq:old-inner}
\lambda f^m_{\left[\frac14,\frac34\right]\times
  \left[\eta_0 m_{\rho_1},\,\overline M_{\rho_1}\right]}
 > \frac{A(\rho_1)}{3/32}
   \left[1+\frac{\rho_1}{\eta_0^{\,7/8}\,(1/2)}\right]^{8/5}
 \approx 2.99153\times10^{-4}
\end{equation}
and
\begin{equation}\label{eq:old-outer}
\lambda f^M_{[0,1]\times\left[0,\,\overline M_{\rho_2}\right]}
 < \frac{A(\rho_2)}{1/8}\,3^{8/7}
 = 8000\cdot 3^{8/7}
 \approx 2.80783\times10^{4},
\end{equation}
where
\[
 \eta_0 m_{\rho_1}=\frac14\!\left(\frac{1}{2500}\right)^{8/5}\approx 9.1461\times10^{-7},
 \qquad
 \overline M_{\rho_1}=4\cdot 2^{8/5}\!\left(1+\tfrac{1}{2500}\right)^{8/5}\approx 12.1335,
\]
\[
 \overline M_{\rho_2}=4\cdot 2^{8/5}\cdot 4^{8/5}\approx 111.4305,
\]
then \cite[Theorem~3.6]{goodrich2025unified} implies, in the special case $b\equiv\textbf{1}$, with the $q=1$ sub-additivity estimate, that problem \eqref{eq2.35}
has at least one positive solution
$u_0\in\widehat V_3\setminus\overline{\widehat V_{1/2500}}$ satisfying the
supremum-norm localisation
\[
 \left(\frac{\rho_1}{(b*\bm1)(1)}\right)^{1/p^+}+\varepsilon(\rho_1,b)
 \le \|u_0\|_\infty
 \le \overline M_{\rho_2}\approx 111.4305.
\]

\subsection*{Hybrid cone framework (Theorem~\ref{thm:existence-concave})}
Suppose that $\lambda$ and $f$ satisfy both
\begin{equation}\label{eq:new-inner}
\lambda f^m_{\left[\frac14,\frac34\right]\times
  \left[\eta_0 m_{\rho_1},\,B^{*}_{\infty,\rho_1}\right]}
 > \frac{A(\rho_1)}{1/16}\left(\frac{\rho_1}{\beta-\alpha}\right)^{8/7}
 \approx 1.29060\times10^{-7}
\end{equation}
and
\begin{equation}\label{eq:new-outer}
\lambda f^M_{[0,1]\times\left[0,\,B^{*}_{\infty,\rho_2}\right]}
 < \frac{A(\rho_2)\,\rho_2^{1/p^+}}
        {\Lambda_{p(\cdot)}\!\left(\int_0^1 G(\cdot,s)\,ds\right)}
 = \frac{1000\cdot 3^{8/7}}{\Lambda_{p(\cdot)}(h)}
 \approx 4.35375\times10^{4},
\end{equation}
where
\[
 \eta_0 m_{\rho_1}=\frac14\!\left(\frac{1}{2500}\right)^{8/5}\approx 9.1461\times10^{-7},
 \qquad
 B^{*}_{\infty,\rho_1}=4\cdot 2^{8/5}\!\left(\frac{1}{2500}\right)^{8/7}\approx 1.5862\times10^{-3},
\]
\[
 B^{*}_{\infty,\rho_2}=4\cdot 2^{8/5}\cdot 3^{8/5}\approx 70.3237,
\]
and $\displaystyle h(t):=\int_0^1 G(t,s)\,ds=\frac12\,t(1-t)$, with
$\Lambda_{p(\cdot)}(h)\approx 0.0806155$ being the Luxemburg gauge of $h$ at the exponent
$p(\cdot)$, computed as the unique $\delta>0$ solving
$\displaystyle\int_0^1\bigl(h(t)/\delta\bigr)^{p(t)}\,dt=1$. Then
Corollary~\ref{cor:b-equiv-one} implies that problem \eqref{eq2.35} has at least one positive
solution
$u_0\in\widehat V_3\setminus\overline{\widehat V_{1/2500}}$ satisfying the
Luxemburg gauge localisation
\[
 \eta_0(\beta-\alpha)^{1/p^-}m_{\rho_1}
 \le \|u_0\|_{L^{p(\cdot)}}
 \le \rho_2^{1/p^-}=3^{8/5}\approx 5.7995.
\]

In light of the preceding, we enumerate the improvements afforded by the present
methodology; throughout, the comparison is the special $b\equiv\bm1$, $q$-free
comparison described above.  Recall that the comparisons enumerated below are depicted graphically in Figures~\ref{fig1}--\ref{fig2} from Section 1.

\begin{enumerate}[label=(\roman*),leftmargin=1.2cm] \label{enumeration}
\item \textbf{Inner rectangle height.}
The admissible $u$-range shrinks from $\overline M_{\rho_1}\approx 12.1335$ to
$B^{*}_{\infty,\rho_1}\approx 1.5862\times10^{-3}$, a reduction of about
$99.99\%$, which makes the condition involving $f^m$ markedly easier to satisfy.
\item \textbf{Outer rectangle height.}
The $u$-range shrinks from $\overline M_{\rho_2}\approx 111.4305$ to
$B^{*}_{\infty,\rho_2}\approx 70.3237$, a reduction of about $36.89\%$, which
makes the condition involving $f^M$ easier to satisfy.
\item \textbf{Direct localisation.}
The older framework localises in the supremum norm, giving
$\|u_0\|_\infty\le \overline M_{\rho_2}\approx 111.4305$, whereas the present
framework localises in the Luxemburg gauge, giving
$\|u_0\|_{L^{p(\cdot)}}\le 3^{8/5}\approx 5.7995$. Measured in the same style as
the convex example's comparison, this is an approximately $94.80\%$ tightening
(about $19\times$ tighter).
\item \textbf{Inner radius threshold.}
Comparing \eqref{eq:old-inner} and \eqref{eq:new-inner}, the threshold on
$\lambda f^m$ falls from about $2.99153\times10^{-4}$ to about
$1.29060\times10^{-7}$, a reduction of about $99.96\%$ (roughly $2318\times$
smaller). The improvement is due entirely to the removal of the subadditive
correction, whose additive $+1$ prevents the older threshold from decaying with the
small inner mass: the non-Green factor falls from
\[
  \bigl[1+\rho_1\eta_0^{-7/8}(\beta-\alpha)^{-1}\bigr]^{8/5}\approx1.00431
\]
to
\[
  \bigl(\rho_1/(\beta-\alpha)\bigr)^{8/7}\approx2.88851\times10^{-4},
\]
a ratio of about $3476.91$. This gain is achieved \emph{despite} the more conservative Green constant
\[
  \inf_{t\in[\alpha,\beta]}\int_\alpha^\beta G(t,s)\ ds=\frac{1}{16},
\]
used here in place of
\[
  \max_{t\in[0,1]}\int_\alpha^\beta G(t,s)\ ds=\frac{3}{32},
\]
which offsets the improvement by a factor of $\frac{3}{2}$; and indeed
$3476.91/1.5\approx2317.94$.
\item \textbf{Outer radius allowance.}
Comparing \eqref{eq:old-outer} and \eqref{eq:new-outer}, the allowance on
$\lambda f^M$ rises from about $2.80783\times10^{4}$ to about
$4.35375\times10^{4}$, an increase of about $55.06\%$, which makes the outer
condition easier to satisfy.
\end{enumerate}

This demonstrates that, even in the strictly concave regime, the hybrid framework
offers a substantial improvement for this class of problem, and that the Luxemburg
gauge $\|\cdot\|_{L^{p(\cdot)}}$ provides a more natural and more accurate
localisation of the solution of \eqref{eq1.1}.
\end{example}

\begin{remark}\label{remark3.5}
Although in this work we have focused on the setting of \emph{existence} of solution, there are several recent papers on \emph{non-existence} of solution in the nonlocal differential equations setting -- e.g., see the recent work by both the first author \cite{goodrich19,goodrich24,goodrich25,goodrich30} and by Shibata \cite{shibata1,shibata2,shibata3}.  In light of the new, Luxemburg-driven methodology introduced in both this present work as well as the companion work \cite{goodrich2025luxemburgnormlocalisationnonlocal}, it could be an interesting avenue for future research to see whether this new methodology can provide enhanced non-existence results in the setting of nonlocal equations with variable growth.
\end{remark}

%\textbf{Acknowledgements.} We gratefully acknowledge the careful reading by the anonymous referee.  Their suggestions for improvement were most helpful.

\textbf{Author Contribution Statement.} Both authors (CSG and GN) contributed equally to the research.  GN wrote an initial version of the paper, which CSG then refined into the final version of the paper; GN created the figures.  Both CSG and GN read and approved the final version of the paper.

\textbf{Funding.} This research received no funding.

\textbf{Data Availability Statement.} This manuscript has no associated data.

\end{document}